\documentclass[UTF-8,reqno]{amsart}
\usepackage{enumerate, bbm, framed}
\usepackage{amssymb,url,color, booktabs}
\usepackage{tikz}
\usepackage{mathrsfs}
\usepackage{dutchcal}
\usepackage{threeparttable}

\usepackage{tikz}
\usepackage{amsmath}

\newtheorem{thrm}{Theorem}[section]
\newtheorem{lemma}[thrm]{Lemma}

\newtheorem{remark}[thrm]{Remark}

\numberwithin{equation}{section}

\usepackage{enumerate}

\usepackage{amssymb}
\usepackage{mathtools}
\usepackage{mathrsfs}
\usepackage{comment}
\usepackage{hyperref}
\usepackage{xcolor}

\def\mI{\mathbb{I} }
\def\E{\mathbb{E} }
\def\P{\mathbb{P} }

\def\R{\mathbb{R} }
\def\mR{\mathbb{R} }
\def\mT{\mathbb{T} }
\def\mE{\mathbb{E} }
\def\mN{\mathbb{N} }

\def\sI{\mathscr{I} }
\def\N{\mathbb{N} }

\def\cN{\mathcal{N} }

\def\e{{\rm e }}
\def\1{{{\mathbbm 1}}}
\def\dif{{\rm d}}

\def\bs{\boldsymbol{s}}
\def\by{\boldsymbol{y}}
\def\bmu{\boldsymbol{\mu}}

\def\p{{\partial}}
\def\div{{\rm div}}
\def\sL{{\mathscr L}}

\def\cM{{\mathcal M} }

\def\<{{\langle}}
\def\>{{\rangle}}
\def\geq{\geqslant}
\def\leq{\leqslant}

\makeatletter

\begin{document}
	\allowdisplaybreaks
	
	\title{Heat kernel estimates for nonlocal kinetic operators}
	\author{Haojie Hou and Xicheng Zhang}
	
	\address{Haojie Hou: School of Mathematics and Statistics, Beijing Institute of Technology, Beijing 100081, China\\	Email: houhaojie@bit.edu.cn}
	
	\address{Xicheng Zhang:
		School of Mathematics and Statistics, Beijing Institute of Technology, Beijing 100081, China\\
		Faculty of Computational Mathematics and Cybernetics, Shenzhen MSU-BIT University, 518172 Shenzhen, China\\
		Email: xczhang.math@bit.edu.cn
	}


	\thanks{
		This work is supported by National Key R\&D program of China (No. 2023YFA1010103) and NNSFC grant of China (No. 12131019)  and the DFG through the CRC 1283 ``Taming uncertainty and profiting from randomness and low regularity in analysis, stochastics and their applications''. }

	\begin{abstract}		
		In this paper, we employ probabilistic techniques to derive sharp, explicit two-sided estimates for the heat kernel of the nonlocal kinetic operator
		\[
		\Delta^{\alpha/2}_v + v \cdot \nabla_x, \quad \alpha \in (0, 2),\ (x,v)\in\mR^{d}\times\mR^d,
		\]
		where \( \Delta^{\alpha/2}_v \) represents the fractional Laplacian acting on the velocity variable \( v \). Additionally, we establish logarithmic gradient estimates with respect to both the spatial variable \( x \) and the velocity variable \( v \).
		In fact, the estimates are developed for more general non-symmetric stable-like operators, demonstrating explicit dependence on the lower and upper bounds of the kernel functions.  These results, in particular, provide a solution to a fundamental problem in the study of \emph{nonlocal} kinetic operators.
		
	\end{abstract}

	\maketitle
	\setcounter{tocdepth}{2}

	\section{Introduction}
	
	Heat kernel estimates are a crucial tool in the study of partial differential equations (PDEs), particularly in the analysis of diffusion processes and heat flow. They provide detailed information about how solutions to diffusion-type equations, such as the heat equation, evolve over time. More generally, heat kernel estimates describe how heat (or probability, in the case of diffusion processes) spreads through space, with applications ranging from geometry to probability theory and mathematical physics (see \cite{CCFI11}). 
	The classical Laplacian heat equation  is given by:
	\[
	\partial_t u = \Delta u,
	\]
	where \( \Delta \) is the Laplacian operator. The fundamental solution to this equation, starting from a point source (Dirac delta function), is the \textit{Gaussian heat kernel}:
	\[
	p_t(x,y) = (4\pi t)^{-d/2} \exp\left(-\frac{|x-y|^2}{4t}\right),
	\]
	where \( x, y \in \mathbb{R}^d \) are spatial points, and \( t > 0 \) is the time variable. 
	
	Beyond the Laplacian heat equation, a second-order differential operator of divergence form typically takes the form:
	\[
	\sL^{\rm div}_{\rm loc} f = \div(a\cdot\nabla f) + \div(b f),
	\]
	and the general non-divergence second-order differential operator takes the form:
	\[
	\sL^{\rm ndiv}_{\rm loc} f =  \text{tr}(a\cdot\nabla^2f) + b \cdot \nabla f,
	\]
	where \( a(x) \) is a symmetric diffusion matrix and \( b(x) \) represents the drift term. These operators often describe diffusion processes with drift, such as stochastic processes like Brownian motion with drift. The heat kernel provides the probability distribution of a particle at time \( t \), starting from point \( x \) and moving to point \( y \).
	The heat kernel estimates for second-order operators are usually given in the Gaussian form of upper and lower bounds for \( p_t(x, y) \), depending on the distance between \( x \) and \( y \), time \( t \), and the regularities of its coefficients. For the divergence form operator \( \sL^{\rm div}_{\rm loc} \), under bounded measurable and uniformly elliptic assumptions, a celebrated result due to Aronson \cite{Ar} provides sharp two-sided Gaussian-type estimates:
	\begin{align}\label{Two1}
		C_0t^{-d/2} \exp\left( -\lambda_0\frac{|x - y|^2}{t} \right) \leq p_t(x, y) \leq C_1t^{-d/2} \exp\left( -\lambda_1\frac{|x - y|^2}{t} \right),
	\end{align}
	where the constants $\lambda_0,\lambda_1>0$ and $C_0,C_1>0$ only depend on the bounds and the dimension.
	It was shown in \cite{FS86} that the above estimates are equivalent to the parabolic Harnack inequality. For the operator \( \sL^{\rm ndiv}_{\rm loc} \) of non-divergence form, when \( a \) is bounded, uniformly elliptic, and H\"older continuous, and \( b \) is bounded and measurable, the two-sided estimates \eqref{Two1} were established in \cite{Fr75}. These two-sided estimates play a crucial role in many analytical problems, and heat kernel estimates remain a central topic in modern mathematical physics. More profound results on heat kernel estimates on manifolds can be found in the monographs \cite{Da89, Gr09}.
	
	Although the heat kernel estimates for local operators \( \sL^{\rm div}_{\rm loc} \) and \( \sL^{\rm ndiv}_{\rm loc} \) are extensively studied in the literature, nonlocal operators have also garnered great interest due to their mathematical challenges and practical applications. In recent years, significant attention has been paid to the study of heat kernel estimates for nonlocal operators, such as the fractional Laplacian \( \Delta^{\alpha/2} \), where \( \alpha \in (0,2) \), defined via Fourier's transform:
	\[
	\widehat{(-\Delta)^{\alpha/2} f}(\xi) = |\xi|^\alpha \hat f(\xi).
	\] 
	The heat equation corresponding to the fractional Laplacian is:
	\[
	\partial_t u = {-(-\Delta)^{\alpha/2} u=:\Delta^{\alpha/2} u.}
	\]
	This operator describes processes with jumps or long-range interactions, typical in Lévy flights and anomalous diffusion (cf. \cite{CMKG08}). The heat kernel for the fractional Laplacian does not have a simple closed-form expression like the Gaussian kernel, except in special cases (e.g., \( \alpha = 1 \)). However, sharp two-sided estimates are established by 
	P\'olya \cite{Po1923} for $d=1$ and Blumenthal and Getoor \cite{BG60} for $d\geq 1$: 
	for \( \alpha \in (0, 2) \), the heat kernel \( p_t(x, y) \) of $\Delta^{\alpha/2}$ satisfies:
	\begin{equation}\label{e:1.3}
		p_t(x, y) \asymp t \left( t^{1/\alpha} + |x - y| \right)^{-d - \alpha},
	\end{equation}
	where \( \asymp \) means that the two sides are comparable up to a constant depending on the dimension \( d \) and \( \alpha \). These estimates capture the spread of heat for nonlocal operators, highlighting the long-range effects of the fractional Laplacian.
	
	Over the last two decades, the study of nonlocal operators has made significant progress, including fine potential theoretical properties and the development of De Giorgi-Nash-Moser-Aronson-type theory for symmetric and nonsymmetric nonlocal operators; see, e.g., \cite{Ko88, Kol, BL2, CK03, Bo-Ja0, CK08, C09, CaS1, KSV,  KK18, KKS21, KW24, BKKL, BSK} and the references therein. In particular, Chen and Kumagai \cite{CK03} showed that, for every \( \alpha \in (0,2) \) and for any symmetric measurable function \( \kappa(x, y) \) on \( \mathbb{R}^d \times \mathbb{R}^d \) with \( \kappa_0 \leq \kappa(x,y) \leq \kappa_1 \), the symmetric nonlocal operator
	\begin{equation}\label{e:1.2}
		\sL^{\rm sym}_{\rm nloc} f(x) = \lim_{\epsilon \to 0} \int_{\{ y \in \mathbb{R}^d : |y - x| > \epsilon \}} (f(y) - f(x)) \frac{\kappa(x, y)}{|x-y|^{d+\alpha}} \, \dif y,
	\end{equation}
	defined in the weak sense, admits a jointly H\"older continuous heat kernel \( p_t(x, y) \) with respect to the Lebesgue measure on \( \mathbb{R}^d \), which satisfies \eqref{e:1.3} for all \( t > 0 \) and \( x, y \in \mathbb{R}^d \), where the implicit constant depends only on \( d, \alpha, \kappa_0 \), and \( \kappa_1 \). The two-sided estimate of heat kernel  and Harnack inequalities for symmetric nonlocal operators on metric measure space can be found in \cite{CPKW22} and references therein.
	Meanwhile, consider the following nonsymmetric nonlocal operators:
	\[
	\sL^{\rm nsym}_{\rm nloc} f(x) = \int_{\mathbb{R}^d}(f(x + y) - f(x) - \1_{|y| \leq 1}y \cdot \nabla f(x) \1_{\alpha \in [1,2)}) \frac{\kappa(x, y)}{|y|^{d+\alpha}} \, \dif y,
	\]
	where \( \kappa_0 \leq \kappa(x,y) \leq \kappa_1 \) is a bi-measurable function and $\kappa(x,-y)=\kappa(x,y)$. 
	When \( \kappa \) is H\"older/Dini continuous in $x$ uniformly in $y$, sharp two-sided estimates of heat kernel were obtained in \cite{Ch-Zh, Ch-Zh2,KKS21, MZ22, CZ23}, see also the references therein.
	We emphasize that the constants appearing in all the above results depend on \( \kappa_0 \) and \( \kappa_1 \) implicitly, rather than explicitly.
	
	The local and nonlocal operators discussed above are non-degenerate in the sense that diffusion occurs uniformly in all directions. However, in many applications, particularly in kinetic theory, one encounters degenerate operators where diffusion is not isotropic. These so-called \textit{kinetic operators} often arise in the study of Boltzmann-type equations and kinetic models, such as the Boltzmann equation and the Fokker–Planck equations with nonlocal interactions, which describe the statistical behavior of a gas (see \cite{C88, Vi02}). A typical kinetic operator takes the form:
	\[
	\sL_{\rm kin} f(x,v) = v \cdot \nabla_x f(x,v) + \mathscr{C}f(x,\cdot)(v),
	\]
	where \( v \in \mathbb{R}^d \) represents the velocity variable, \( x \in \mathbb{R}^d \) is the spatial variable, and \( \mathscr{C} \) is an integral-differential operator, such as a collision operator. The solution to a kinetic equation describes the evolution of a probability distribution in phase space \( (x, v) \) over time (see \cite{C88}). It models dynamics that involve both random jumps in velocity (captured by \( \mathscr{C} \)) and transport in space (captured by the drift term \( v \cdot \nabla_x \)).
	A toy model of kinetic operators is of the form (see \cite{A12}):
	\[
	\Delta^{\alpha/2}_v + v \cdot \nabla_x, \quad \alpha \in (0,2].
	\]
	The coupling between space and velocity introduces significant challenges in the analysis of heat kernel estimates for kinetic operators. When \( \alpha = 2 \), in terms of the Brownian diffusion equation, Kolmogorov \cite{Ko34} first wrote down the fundamental solution of \( \Delta_v + v \cdot \nabla_x \) as
	$$
	p_t(x_0,v_0,x,v) = \left(\frac{\sqrt{3}}{2\pi t^2}\right)^d \exp\left\{-\frac{|x - x_0 - t v_0|^2}{2t} + \frac{3(x - x_0 - t v_0)(v - v_0)}{2t^2} - \frac{3|v - v_0|^2}{2t^3}\right\}.
	$$
	From this expression, it is easy to derive the following two-sided estimate:
	\begin{align*}
		&\left(\frac{\sqrt{3}}{2\pi t^2}\right)^d\exp\left\{-\frac{4+\sqrt{13}}{4}\left(\frac{|x - x_0 - t v_0|^2}{t} + \frac{|v - v_0|^2}{t^3}\right)\right\} \leq p_t(x_0,v_0,x,v) \\
		&\qquad \leq \left(\frac{\sqrt{3}}{2\pi t^2}\right)^d \exp\left\{-\frac{4-\sqrt{13}}{4}\left(\frac{|x - x_0 - t v_0|^2}{t} + \frac{|v - v_0|^2}{t^3}\right)\right\}.
	\end{align*}

	Recent research has made progress in deriving sharp two-sided heat kernel estimates for more general local kinetic operators. In \cite{DM10}, Delarue and Menozzi studied Gaussian-type two-sided estimates for general second-order hypoelliptic operators. More recently, the authors in \cite{CMPZ23} and \cite{RZ24} investigated heat kernel estimates for the density of kinetic SDEs with rough coefficients using probabilistic techniques. These estimates are crucial for understanding the behavior of solutions to kinetic equations and for obtaining precise gradient bounds in both space and velocity.
	It is worth noting that kinetic SDEs, also known as Langevin diffusion processes, have been extensively studied in terms of their long-time behavior and convergence to a stationary measure (see \cite{MSH02, HN04, Vi09}).
	
	As previously discussed, in the context of heat kernel estimates for kinetic operators, one often encounters \emph{non-local} operators, particularly those with \( \alpha \in (0,2) \), acting on the velocity variable. This non-locality introduces additional complexity compared to the local case. On one hand, $L^p$-maximal regularity and Schauder estimates for non-local kinetic operators have been studied in \cite{Ch-Zh4, HMP19} and \cite{HWZ19, IS18}, respectively. On the other hand, it is well known through Fourier analysis that the heat kernel of a non-local kinetic operator can be expressed as
	\[
	p_t(x,v) := p_t(0,0,x,v) = \int_{\mathbb{R}^{2d}} \mathrm{e}^{\mathrm{i}(x \cdot \xi + v \cdot \eta)} \mathrm{e}^{-\int_0^1 |s\xi + \eta|^\alpha \, \mathrm{d}s} \, \dif\xi \dif\eta.
	\]
	Although the above expression appears simple, since the seminal work of P\'olya \cite{Po1923} and
	Blumenthal and Getoor \cite{BG60}, obtaining two-sided asymptotic estimates for \( p_t(x,v) \) as \( x, v \to \infty \) has remained an open problem for a considerable time.
	
	The purpose of this paper is to provide a conclusive solution to this question by establishing sharp, explicit, two-sided asymptotic estimates for \( p_t(x,v) \) as \( (x,v) \to \infty \). It is anticipated that these two-sided estimates will prove valuable in the future analysis of the behavior of solutions to the Boltzmann equation (cf. \cite{IS22}).
	
	\subsection{Main result}
	In this section, we present our main result. Before doing so, we first introduce the kinetic operators under consideration and outline the primary assumptions. Specifically, we consider general stable-like non-local operators, which are commonly used to study kinetic operators with variable coefficients. Let \(\alpha \in (0, 2)\) and \(\kappa: \mathbb{R}^d \to (0, \infty)\) be a measurable function satisfying
	\begin{align} \label{Ka1}
		0 < \kappa_0 \leq \kappa(x) \leq \kappa_1, \quad \int_{\{|\omega|=r\}}\omega \kappa(\omega)\dif \omega = 0, \quad \forall r>0.
	\end{align}
	For simplicity of notation, we define
	$$
	\nu(\dif x) := \frac{\kappa(x)}{|x|^{d + \alpha}} \dif x,
	$$
	which is called the L\'evy measure. The second condition in \eqref{Ka1} is equivalent to
	\begin{align} \label{Ka11}
		\int_{R_0\leq |x|\leq R_1}x \, \nu(\dif x) = 0, \quad 0<R_0\leq R_1<\infty.
	\end{align}
	Let \(L_t = L^\nu_t\) be a \(d\)-dimensional L\'evy process with characteristic exponent
	\begin{align}\label{LE1}
		\psi(\xi) = \psi_\nu(\xi) = \int_{\mathbb{R}^{d} \setminus \{0\}} \left(1 - \e^{\mathrm{i} \xi \cdot x} + \mathrm{i} \xi \cdot x\1_{\{|x|\leq 1\}}\right)\nu(\mathrm{d}x).
	\end{align}
	By the L\'evy-Khintchine formula (see \cite{Sato}), we have
	\[
	\E \left(\e^{\mathrm{i} \xi \cdot L_t}\right) = \e^{-t\psi(\xi)}, \quad t \geq 0, \, \xi \in \mathbb{R}^d.
	\]
	The condition \eqref{Ka11} ensures that \(L_t\) has zero mean; otherwise, \(L_t\) would include a linear drift term. Furthermore, \eqref{Ka11} implies the following scaling property: for \(\lambda>0\),
	\begin{align}\label{MC1}
		(L^\nu_{\lambda t})_{t \geq 0}\stackrel{(d)}{=}\lambda^{1/\alpha}(L^{\nu_\lambda}_t)_{t \geq 0}, \quad \nu_\lambda(\dif x) := \frac{\kappa(\lambda^{1/\alpha}x)}{|x|^{d + \alpha}} \dif x.
	\end{align}
	Indeed, by a change of variables, we have
	\begin{align*}
		\mathbb{E} \left(\e^{\mathrm{i} \xi \cdot L^\nu_{\lambda t}}\right) = \e^{-\lambda t\psi_\nu(\xi)} = \e^{- t\psi_{\nu_\lambda}(\lambda^{1/\alpha}\xi)} = \mathbb{E} \left(\e^{\mathrm{i}\xi \cdot \lambda^{1/\alpha} L^{\nu_\lambda}_{t}}\right).
	\end{align*}
	
	For a given \(z_0 = (x_0, v_0)\), consider the Markov process \(Z_t(z_0)\) in \(\mathbb{R}^{2d}\) starting from \(z_0\),
	\[
	Z_t(z_0) := \left(x_0 + tv_0 + \int_0^t L_s \mathrm{d}s, v_0 + L_t\right) =: (x_0 + tv_0 + X_t, v_0 + V_t).
	\]
	By Itô's formula, we have (see \cite{Ch-Zh4})
	\[
	\partial_t \mathbb{E} f(Z_t(z_0)) = \mathbb{E} \mathscr{K}_\nu f(Z_t(z_0)),
	\]
	where \(\mathscr{K}_\nu\) is the non-local kinetic operator defined by
	\[
	\mathscr{K}_\nu f(x, v) = \mathscr{L}^{\nu} f(x, \cdot)(v) + v \cdot \nabla_x f(x, v),
	\]
	and \(\mathscr{L}^{\nu}\) is the infinitesimal generator of \(L\), given by
	\begin{align*}
		\mathscr{L}^{\nu} f(v) &= \int_{\mathbb{R}^d} \left(f(v + w) - f(v) - \1_{\{|w|\leq 1\}}w \cdot \nabla f(v)\right) \nu(\dif w).
	\end{align*}
	As shown in \cite[Lemma 2.5]{Ch-Zh4}, under condition \eqref{Ka1} (although \cite{Ch-Zh4} considers the symmetric case, the result can be extended to the general case \eqref{Ka1}, see Lemma \ref{lemma:small-jump} below), for any \(t > 0\), the process \(Z_t(z_0)\) admits a smooth density \(p_t(z_0, z)\), meaning that for any bounded measurable   \(f: \mathbb{R}^{2d} \to \mathbb{R}\),
	\begin{align}\label{DY1}
		\int_{\mathbb{R}^{2d}} f(z) p_t(z_0, z) \, \mathrm{d}z = \mathbb{E} f(Z_t(z_0)).
	\end{align}
	In particular, the density \(p_t(z_0, \cdot)\) of \(Z_t(z_0)\) is the heat kernel of the kinetic operator \(\mathscr{K}_\nu\), i.e.,
	$$
	\p_tp_t(z_0, z) = \mathscr{K}_\nu^* p_t(z_0, \cdot)(z).
	$$
	Moreover, by \eqref{DY1}, we easily see that
	\begin{align} \label{DQ1}
		p_t(z_0, z) = p_t(0, z - \theta_t z_0), \quad \theta_t z := (x + tv, v) \in \mathbb{R}^{2d}.
	\end{align}
	Hence, to obtain two-sided estimates of \(p_t(z_0, z)\), it suffices to consider \(p_t(0, z) := p_t^\nu(z)\), which is the density of 
	\begin{align} \label{DQ2}
		Z^\nu_t :=  Z_t(0) = \left(\int_0^t L_s \mathrm{d}s, L_t\right) = (X_t, V_t) =: Z_t.
	\end{align}
	Here is a figure illustrating \((Z^\nu_t)_{t \in [0,1]}\) for \(\kappa(x) = 1\) and \(\alpha = 1.5\).
	\begin{figure}[ht]
		\centering
		\includegraphics[width=4in, height=2in]{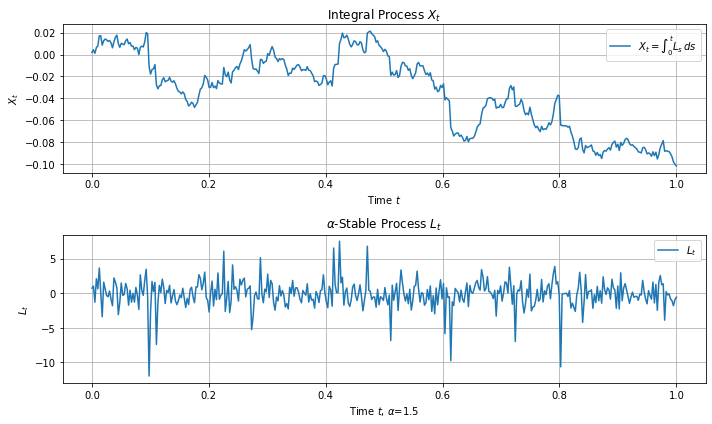}
		\label{fig:sd21}
	\end{figure}
	
	To state the main results, for \(z=(x,v)\in\mathbb{R}^{2d}\) and \(t>0\), we introduce the following notations:
	\begin{align}\label{TT1}
		\mT^\alpha_t z := (t^{-\frac1\alpha-1}x,t^{-\frac1\alpha}v), \quad \Gamma_t z := x - tv.
	\end{align}
	
	\begin{thrm}\label{thm1}
		We assume \eqref{Ka1} holds and for \(\beta>1\), define
		\begin{align}\label{Nbeta}
			\mathcal{N}_\beta(z) := \frac{1}{(1+|z|)^{1+\beta}}\left(\inf_{s\in [0,1]}|\Gamma_s z| +1 \right)^{1-\beta}.
		\end{align}
		\begin{enumerate}[(i)]
			\item  {\bf (Lower bound estimate):} For \(c_0^* := \int_{|y|>1/3} \frac{\dif y}{|y|^{d+\alpha}}\), there exists a constant \(C = C(d,\alpha) > 0\) such that for all \(z\in\mathbb{R}^{2d}\) and \(t>0\),
			\begin{align}\label{TSE1}
				p_t^\nu(z) \gtrsim_C  \kappa_0^{-\frac{2d}\alpha} \e^{-c_0^* \left(\frac{\kappa_1}{\kappa_0}\right)}   t^{-\frac{2d}{\alpha} -d} \mathcal{N}_{d+\alpha}\left(\kappa_0^{-\frac1\alpha}\mT^\alpha_{t}z\right).
			\end{align}
			
			\item {\bf (Upper bound and gradient estimates):} For any \(j_x, j_v\in\mathbb{N}_0:=\{0\}\cup\mN\), there exists a constant \(C = C(d,\alpha,j_x,j_v) > 0\) such that for all \(z \in \mathbb{R}^{2d}\) and \(t>0\),
			\begin{align}\label{TSE3}
				|\nabla^{j_x}_x \nabla^{j_v}_v p_t^\nu(z)| \lesssim_C \kappa_0^{-\frac{2d+j_x+j_v}{\alpha}} \left(\frac{\kappa_1}{\kappa_0}\right)^{3+4d+3\alpha} t^{-\frac{2d+j_x+j_v}{\alpha} - j_x - d} \mathcal{N}_{d+\alpha}\left(\kappa_0^{-\frac1\alpha}\mT^\alpha_{t}z\right).
			\end{align}
		\end{enumerate}
	\end{thrm}
	
	\begin{remark}\rm
		For simplicity of presentation, we consider here only the time-homogeneous case. 
We believe that  our heat kernel estimates can be extended to the time-dependent kernel \( \kappa(t, x) \), as considered in \cite{Ch-Zh4}, without essential difficulty. In a future work, we shall consider the time and spatial inhomogeneous nonlocal kinetic operators.
	\end{remark}
	
	\begin{remark}\rm	
		In the gradient estimate \eqref{TSE3}, the explicit dependence on \(\kappa_0\) and \(\kappa_1\) is useful when we consider variant and growth coefficients.
		By \eqref{TSE1} and \eqref{TSE3}, we immediately have the following logarithmic gradient estimate: for some \(C=C(d,\alpha)>0\),
		\[
		t^{1/\alpha+1}|\nabla_x\log p_t^\nu(z)|+t^{1/\alpha}|\nabla_v\log p_t^\nu(z)|\lesssim_C \kappa_0^{-\frac1\alpha}\e^{c_0^* \left(\frac{\kappa_1}{\kappa_0}\right)}\left(\frac{\kappa_1}{\kappa_0}\right)^{3+4d+3\alpha} .
		\]
	\end{remark}
	
	\begin{remark}\label{Remark}\rm
		Unlike the local case, it is difficult to conjecture the correct form of the two-sided bounds. Furthermore, there does not seem to be an intuitive reason why the two-sided estimates of \( p_1^\nu(z) \) should resemble \( \mathcal{N}_{d+\alpha}(z) \). 
		Let \(\bar{v} := v/|v|\). Noting that 
		\[
		|x - sv|^2 = (s|v| - \langle x, \bar{v} \rangle)^2 + |x|^2 - \langle x, \bar{v} \rangle^2,
		\]
		we have
		\begin{align}\label{AA6}
			\mathcal{N}_\beta(z)
			= \frac{1}{(1+|z|)^{1+\beta}}
			\left\{
			\begin{aligned}
				&(1 + |x|)^{1-\beta}, & \langle x, v \rangle \leq 0,\\
				&(1 + \sqrt{|x|^2 - \langle x, \bar{v} \rangle^2})^{1-\beta}, & 0 < \langle x, v \rangle \leq |v|^2,\\
				&(1 + |x - v|)^{1-\beta}, & \langle x, v \rangle > |v|^2.
			\end{aligned}
			\right.
		\end{align}
		Here is a picture of the above three cases:
		
		\begin{tikzpicture}[scale=1.5]
			
			\draw[->] (-3, 0) -- (3, 0) node[right] {$x$};
			\draw[->] (0, -2) -- (0, 2) node[above] {$v$};
			
			\fill[red!30, opacity=0.6] (-3,0) -- (0,0) -- (0,2) -- (-3,2)--cycle;
			\fill[red!30, opacity=0.6] (3,-2) -- (3,0) -- (0,0) -- (0,-2)--cycle;
			\draw [red!30, fill=red!30] (4,0.9) -- (3.75,0.9) -- (3.75,1.1) -- (4,1.1);
			\node at (1.5, -1) {$(1+|x|)^{1-\beta}$};
			\node at (-1.5, 1) {$(1+|x|)^{1-\beta}$};
			\node at (5.15, 1) { $\< x, v\> < 0$};

			\fill[green!30, opacity=0.6] (0,0) -- (3,0) -- (3,2) -- cycle;
			\fill[green!30, opacity=0.6] (0,0) -- (-3,0) -- (-3,-2) -- cycle;
			\node at (1.75, 0.25) {$(1+\sqrt{|x|^2-\<x,\bar v\>^2})^{1-\beta}$};
			\node at (-1.75, -0.25) {$(1+\sqrt{|x|^2-\<x,\bar v\>^2})^{1-\beta}$};
			\node at (5.5, 0) {$0 \leq \< x, v\>  \leq |v|^2$};
			\draw [green!30, fill=green!30] (4,-0.1) -- (3.75,-0.1) -- (3.75,0.1) -- (4,0.1);

			\fill[blue!30, opacity=0.6] (0,0) -- (3,2) -- (0,2) -- cycle;
			\fill[blue!30, opacity=0.6] (0,0) -- (-3,-2) -- (0,-2) -- cycle;
			\node at (5.3, -1) { $\< x, v\>  > |v|^2$};
			\node at (1, 1.2) {$(1+|x-v|)^{1-\beta}$};
			\node at (-1, -1.4) {$(1+|x-v|)^{1-\beta}$};
			\draw [blue!30, fill=blue!30] (4,-1.1) -- (3.75,-1.1) -- (3.75,-0.9) -- (4,-0.9);
			
			\node at (1.75,1.75) {$(x, v)$ plane};
			
		\end{tikzpicture}
		
	\end{remark}
	
	The following lemma, proven in Section 3, provides some useful integral estimates for \(\cN_\beta\).
	\begin{lemma}\label{Le11}
		Let $\beta > d$ and $q\geq 0$. If $q\geq 2\beta-d$, then for any $v\in\mathbb{R}^d$,
		\begin{align}\label{AG6}
			\int_{\mathbb{R}^{d}} |x|^q\mathcal{N}_\beta(x,v)\,\mathrm{d}x = \infty.
		\end{align}
		If $q<2\beta-d$, then there is a constant $C=C(q,\beta,d)\geq 1$ such that
		\begin{align}\label{Ineq-in-remark2}
			\int_{\mathbb{R}^{d}} |x|^q\mathcal{N}_\beta(x,v)\,\mathrm{d}x \asymp_C (1+|v|)^{q-\beta}.
		\end{align}
	\end{lemma}
	
	\begin{remark}\rm
		By \eqref{Ineq-in-remark2} and \eqref{TSE3}, we immediately obtain that for \(q_x, q_v \in [0, \alpha)\) with \(q_x + q_v < \alpha\),
		\begin{align*}
			&\int_{\mathbb{R}^{2d}} |x|^{q_x} |v|^{q_v} |\nabla^{j_x}_x \nabla^{j_v}_v p_t^\nu(x, v)| \,\mathrm{d}x \,\mathrm{d}v \lesssim_C t^{\frac{(q_x - j_x)(1 + \alpha) + (q_v - j_v)}{\alpha}}
			\int_{\mathbb{R}^{2d}} |x|^{q_x} |v|^{q_v} \mathcal{N}_{d + \alpha}(x, v) \,\mathrm{d}x \,\mathrm{d}v \\
			&\qquad \lesssim_C t^{\frac{(q_x - j_x)(1 + \alpha) + (q_v - j_v)}{\alpha}} 
			\int_{\mathbb{R}^{d}} |v|^{q_v} (1 + |v|)^{q_x - d - \alpha} \,\mathrm{d}v  \lesssim_C t^{\frac{(q_x - j_x)(1 + \alpha) + (q_v - j_v)}{\alpha}}.
		\end{align*}
		This is the key result proven in \cite[Lemma 2.5]{Ch-Zh4}.
	\end{remark}
	
	\begin{remark}\rm
		For any \(q \in [0, 2\alpha + d)\), by \eqref{DQ2}, \eqref{TSE1}, and \eqref{AG6}, we have the following asymptotic estimate for the conditional density:
		\[
		\mathbb{E} \left(|X_t|^q \mid V_t = v\right) \asymp t^q (t^{1/\alpha} + |v|)^q, \quad t > 0, \ v \in \mathbb{R}^d.
		\]
		Indeed, from the definition, along with \eqref{TSE1} and \eqref{Ineq-in-remark2}, we have
		\[
		\mathbb{E} \left(|X_t|^q \mid V_t = v\right) = \frac{\int_{\mathbb{R}^d} |x|^q p_t^\nu(x, v) \, \mathrm{d}x}{\int_{\mathbb{R}^d} p_t^\nu(x, v) \, \mathrm{d}x}
		\asymp \frac{\int_{\mathbb{R}^d} |x|^q \mathcal{N}_{d+\alpha}\left(\mT^\alpha_t z\right) \, \mathrm{d}x}{
			\int_{\mathbb{R}^d} \mathcal{N}_{d+\alpha}\left(\mT^\alpha_t z\right) \, \mathrm{d}x} \asymp t^q (t^{1/\alpha} + |v|)^q.
		\]
	\end{remark}
	
	\subsection{Related work}
	In recent years, the study of nonlocal kinetic operators, particularly regarding Harnack inequalities and heat kernel estimates, has garnered significant attentions. In \cite{HRZ24}, Hao, R\"ockner, and Zhang investigated the well-posedness of nonlinear and nonlocal kinetic Fokker-Planck equations with measure-valued initial values. The singular interaction kernels they considered include the Newtonian potential, Coulomb potential, and even the Riesz potential. 
	In \cite{AIN1, AIN2}, Auscher, Imbert, and Niebel developed a comprehensive theory for weak solutions to nonlocal kinetic equations with rough coefficients. Specifically, they established results on the regularity, existence, and uniqueness of weak solutions. While, Kassmann and Weidner \cite{KW2} constructed a counterexample demonstrating the failure of the Harnack inequality for the kinetic operator $v \cdot \nabla_x + \Delta^{\alpha/2}$, highlighting a key difference compared to local kinetic operators (see \cite{GIMA19}).
	On the other hand, for a class of nonlocal hypoelliptic equations in divergence form, Loher \cite{Lo24} proved a strong Harnack inequality and derived rough bounds for the fundamental solution. 
	
	After this paper was posted on arXiv, Grube informed us of his preprint \cite{Gru24}. Using the Fourier transform, he established the following two-sided estimate for the heat kernel of $v \cdot \nabla_x + \Delta^{\alpha/2}$ when $d = 1$ (see \cite[Theorem 1.1]{Gru24}):
	\begin{align}\label{DD1}
		p_1(x,v) \asymp \frac{1}{(1+|x|+|v|)^{2+\alpha}} \frac{1}{(1+(|2x-v|-|v|)_+)^\alpha},
	\end{align}
	where $a_+:=a\vee 0$ for $a\in\mR$.
	When $\alpha = 1$, Grube also derived an explicit formula for $p_1(x,v)$ (see \cite[Theorem 5.1]{Gru24}).
	
	To make a comparison, for $z = (x,v)$, let $\bar{x} := \frac{x}{|x|}$, $\bar{v} := \frac{v}{|v|}$, and $\omega_z := \langle \bar{x}, \bar{v} \rangle$. Noting that
	\[
	|x-v|^2 = |x-\langle x, \bar{v} \rangle \bar{v}|^2 + |\langle x, \bar{v} \rangle \bar{v} - v|^2 = 2|x|^2(1-\omega_z^2) + ||x|\omega_z - |v||^2,
	\]
	where we adopt the convention $\bar{0} := 0$, it follows that
	\[
	1 + |x|\sqrt{1 - \omega_z^2} + \left( \big|2|x| \omega_z - |v| \big| - |v| \right)_+ 
	\asymp
	\begin{cases}
		1 + |x|, & \omega_z \leq 0, \\
		1 + |x|\sqrt{1 - \omega_z^2}, & 0 < \omega_z \leq \frac{|v|}{|x|}, \\
		1 + |x - v|, & \frac{|v|}{|x|} < \omega_z \leq 1.
	\end{cases}
	\]
	Thus, by \eqref{AA6}, we obtain
	\begin{align}
		\mathcal{N}_\beta(z) \asymp \frac{1}{(1+|z|)^{1+\beta}} \left( 1 + |x|\sqrt{1 - \omega_z^2} + \left( \big|2|x| \omega_z - |v| \big| - |v| \right)_+ \right)^{1-\beta}.
	\end{align}
	Compared to \eqref{DD1}, there is an additional term $ |x|\sqrt{1 - \omega_z^2} $, which vanishes for $x,v \neq 0$ when $d = 1$. 
	In particular, when $d = 1$, our two-sided estimate $\mathcal{N}_{1+\alpha}(z)$ coincides with \eqref{DD1}.
	
	\subsection{Strategy of proof}
	Motivated by \cite{Wa} and following \cite{Ch-Zh2}, we apply L\'{e}vy-It\^o's decomposition to decompose \( Z_t \) into a small jump part and a large jump part:
	\[
	Z_t = Z_t^{(0)} + Z_t^{(1)},
	\]
	where 
	\begin{align*}
		Z_t^{(i)} := (X_t^{(i)}, V_t^{(i)}) = \left( \int_0^t L_s^{(i)} \, \mathrm{d}s, L_t^{(i)} \right), \quad i = 0, 1,
	\end{align*}
	and \( L_t^{(i)}, i=0,1 \) are two independent L\'{e}vy processes with characteristic exponent \( \psi_{  \nu^{(i)} } \), \(i=0,1\) (see \eqref{LE1}), and
	\[
	\nu^{(0)}(\mathrm{d}x) := \1_{\{|x| \leq 1\}} \frac{\kappa(x) \, \mathrm{d}x}{|x|^{d+\alpha}}, \quad  \nu^{(1)}(\mathrm{d}x) := \1_{\{|x| > 1\}} \frac{\kappa(x) \, \mathrm{d}x}{|x|^{d+\alpha}}.
	\]
	Under assumption \eqref{Ka1}, the small jump part \( Z_t^{(0)} \) has a smooth, strictly positive density \( q_t(z) \) that decays at any polynomial rate, which allows us to handle it more easily using Fourier transform techniques (see Lemmas \ref{lemma:small-jump} and \ref{lemma:small-jump-lowerbound}). In particular, since \( Z_t^{(0)} \) and \( Z_t^{(1)} \) are independent, one has
	\begin{align}\label{Identity}
		p^\nu_t(z) = \mathbb{E}\left( q_t(z - Z_t^{(1)}) \right).
	\end{align}
	Thus, the large jump part \( Z_t^{(1)} \) governs the decay rate in the two-sided estimate. We must point out that although \( L_t^{(1)} \) is a compound Poisson process (see \cite{Sato}), finding a suitable explicit function \( \mathcal{E}: (0,\infty)\times\mathbb{R}^{2d} \to (0, \infty) \) to obtain a sharp two-sided estimate of the form
	\[
	\mathbb{P}(|Z_t^{(1)} - z| \leq 1) \asymp \mathcal{E}(t,z),
	\]
	remains a challenge due to the degenerate nature of the process (see Lemma \ref{lemma2} below).
	
	In the next two sections, we provide complete proofs of Theorem \ref{thm1} and Lemma \ref{Le11}. Throughout this paper, by \( A \asymp B \), \( A \lesssim B \), and \( A \gtrsim B \), we mean that for some unimportant constant \( C \geq 1 \) that only depends on \( d,\alpha \),
	\[
	C^{-1} B \leq A \leq C B, \quad A \leq C B, \quad \text{and} \quad A \geq C^{-1} B.
	\]

	\section{Proof of Theorem \ref{thm1}}

	We first recall a well-known result. For reader's convenience, we provide detailed proofs.
	\begin{lemma}\label{lemma1}
		Let $L_t$ be a L\'{e}vy process with characteristic exponent $\psi$ and $Z_t$ be as in \eqref{DQ2}. Then for any $\xi, \eta\in \R^d$ and any $t\geq 0$,  it holds that 
		\begin{align}\label{Charac-of-Z}
			\E\left(\e^{\mathrm{i} (\xi, \eta)\cdot Z_t}\right)= \exp\left\{-\int_0^t \psi \left(s\xi +\eta\right)\mathrm{d}s\right\}.
		\end{align}
	\end{lemma}
	\begin{proof} 
		Given $t>0$, for each $n\in \N$ and $k=1,\cdots,n$, 
		define $\Delta_k:= L_{kt/n}-L_{(k-1)t/n}$ and 
		\[
		{Z}_t^{( n)}:= \left(\frac{t}{n}\sum_{k=1}^n L_{kt/n}, L_t\right) = \left(\frac{t}{n}\sum_{k=1}^n \sum_{j=1}^k \Delta_j, \sum_{k=1}^n \Delta_k\right).
		\]
		Then for each $t>0$ and $n\in \N$,
		\begin{align*}
			\E\left(\e^{\mathrm{i} (\xi, \eta)\cdot{Z}_t^{( n)}}\right)& 
= \E\left(\exp\left\{\mathrm{i}\left( \frac{t}{n}\sum_{k=1}^n \sum_{j=1}^k  \xi\cdot \Delta_j + \sum_{k=1}^n \eta \cdot \Delta_k \right)\right\}\right) \nonumber\\
			& =  \E\left(\exp\left\{\mathrm{i}\left( \sum_{k=1}^n  \left(  \frac{(n-k+1)t\xi}{n}+ \eta \right)\cdot \Delta_k \right)  \right\}\right) \nonumber\\
			& = \exp\left\{- \frac{t}{n}\sum_{k=1}^n \psi\left( \frac{(n-k+1)t\xi}{n}+ \eta \right) \right\},
		\end{align*}
		where in the last equality we used the fact that $\{\Delta_k, 1\leq k\leq n\}$ is independent and the definition \eqref{LE1} of $\psi$.  Since $L_t$ is right-continuous with left hand limit almost surely, we see that the path of $L$ is almost surely Riemann integrable. Therefore,
		\begin{align*}
			&\E\left(\e^{\mathrm{i} (\xi, \eta)\cdot Z_t}\right) = \lim_{n\to\infty} \E\left(\e^{\mathrm{i} (\xi, \eta)\cdot {Z}_t^{( n)}}\right) \nonumber\\
			&  = \lim_{n\to\infty} \exp\left\{- \frac{t}{n}\sum_{k=1}^n \psi\left(\frac{(n-k+1)t\xi}{n}+ \eta \right) \right\}= \exp\left\{-\int_0^t \psi \left((t-s)\xi+\eta\right)\mathrm{d}s\right\},
		\end{align*}
		which implies \eqref{Charac-of-Z}. 
	\end{proof} 
	
	For \( \lambda, t > 0 \), we have the following scaling property:
	\begin{align} \label{Scaling}
		p^{\lambda\nu}_{t}(x,v) = \lambda^d p^\nu_{\lambda t}(\lambda x,v) = \lambda^d(\lambda t)^{-\frac{2d}{\alpha} - d} p^{\nu_{\lambda t}}_1(\mT^\alpha_{\lambda t} (\lambda x,v)),
	\end{align}
	where \( \mT^\alpha_t \) is defined by \eqref{TT1} and
	\[
	\nu_{\lambda t}(\dif x) = \frac{\kappa((\lambda t)^{1/\alpha}x)}{|x|^{d+\alpha}} \dif x.
	\]
	Indeed, by Fourier's transform and the change of variable, we have
	\begin{align*}
		\int_{\mR^{2d}} \e^{\mathrm{i} (\xi, \eta) \cdot z} p^{\lambda\nu}_{t}(z) \, \dif z
		&= \exp\left\{-\lambda \int_0^t \psi \left(s\xi + \eta\right)\, \mathrm{d}s\right\}
		= \exp\left\{-\int_0^{\lambda t} \psi \left(s\xi/\lambda + \eta\right)\, \mathrm{d}s\right\}\\
		&= \int_{\mR^{2d}} \e^{\mathrm{i} (\xi/\lambda, \eta) \cdot z} p^{\nu}_{\lambda t}(z) \, \dif z
		= \lambda^d \int_{\mR^{2d}} \e^{\mathrm{i} (\xi \cdot x + \eta \cdot v)} p^{\nu}_{\lambda t}(\lambda x,v)\, \dif x \, \dif v,
	\end{align*}
	which gives the first equality. The second equality follows from \eqref{MC1} (see \cite{Ch-Zh4}). 
	
	Therefore, by choosing \( \lambda = \kappa_0 \), we shall assume
	\begin{align}\label{Con3}
		1 \leq \kappa(x) \leq \kappa_1,\ \ \int_{\{|\omega|=r\}} \omega \kappa(\omega) \, \dif \omega = 0,\ \ \forall r > 0.
	\end{align}
	Under \eqref{Con3}, to prove Theorem \ref{thm1}, it suffices to consider \( p^{\nu_t}_1 \). By \eqref{Identity}, we will separately estimate \( q_1 \) and \( Z_1^{(1)} \). 
	We begin by demonstrating the following regularity estimates, which state that the small jump part possesses a smooth density that decays at any polynomial rate.
	
	\begin{lemma}\label{lemma:small-jump}
		Under \eqref{Con3}, for any \( j \in \mN_0 \) and \( \beta \geq 0 \), there exists a constant \( C = C(j,\beta,d,\alpha) > 0 \) such that for all \( z \in \R^{2d} \),
		\[
		|\nabla^j q_1(z)| \lesssim_C \kappa^\beta_1 (1 + |z|)^{-\beta}.
		\]
	\end{lemma}
	\begin{proof}
		Let \( \psi_{ \nu^{(0)}} \) be the characteristic exponent of \( L^{(0)}_t \) defined as in \eqref{LE1}. 
		Then, it is easily seen that under \eqref{Con3}, 
		\begin{align} \label{MC2}
			\mbox{Re}(\psi_{ \nu^{(0)}}(\xi)) \geq \int_{\mathbb{R}^{d} \setminus \{0\}} \left(1 - \cos(\xi \cdot x)\right) \frac{\1_{\{|x| \leq 1\}} \, \dif x}{|x|^{d + \alpha}}.
		\end{align}
		For simplicity of notation, we write
		\[
		w = (\xi, \eta) \in \mR^{2d}, \quad \phi(w) := \int_0^1 \psi_{ \nu^{(0)}}(s\xi + \eta) \, \dif s.
		\]
		By \eqref{MC2} and \cite[(2.37) and (2.38)]{Ch-Zh4}, for any \( j \in \mN_0 \), there exist constants \( C, c > 0 \) only depending on \( d, \alpha \) such that for all \( w \in \R^{2d} \),
		\begin{align} \label{eq20}
			\mbox{Re}(\phi(w)) \geq c \left( |w|^2 \land |w|^\alpha \right), \quad |\nabla^j \phi(w)| \lesssim_C \kappa_1 \left( |w|^j + 1 \right).
		\end{align}
		By Fourier's inverse transform and Lemma \ref{lemma1}, we have
		\[
		q_1(z) = (2\pi)^{-2d} \int_{\R^{2d}} \e^{\mathrm{i} w \cdot z} \e^{-\int_0^1 \psi_{ \nu^{(0)}}(s\xi + \eta) \, \dif s} \, \dif w = (2\pi)^{-2d} \int_{\R^{2d}} \e^{\mathrm{i} w \cdot z} \e^{-\phi(w)} \, \dif w.
		\]
		Therefore, for any \( w = (w_1, \dots, w_{2d}) \), \( z = (z_1, \dots, z_{2d}) \in \R^{2d} \) and
		\( j_1, \dots, j_{2d} \in \N_0 \) with \( \sum_{i=1}^{2d} j_i = j \), we have
		\[
		(1 + |z|^2)^m \nabla^{j_1}_{z_1} \cdots \nabla^{j_{2d}}_{z_{2d}} q_1(z) = (2\pi)^{-2d}  {\rm i}^j  \int_{\R^{2d}}
		\e^{\mathrm{i} w \cdot z} \prod_{i=1}^{2d} w_i^{j_i} \times \left( (1 - \Delta_w)^m \e^{-\phi(w)} \right) \, \dif w,
		\]
		where \( \Delta_w \) stands for the Laplacian and \( m \in \mN_0 \).
		By \eqref{eq20}, \( |\e^{-\phi(w)}| = \e^{-\mbox{Re}(\phi(w))} \), and the elementary chain rule, we get for some \( C = C(m, j, d, \alpha) > 0 \),
		\[
		(1 + |z|^2)^m \left| \nabla^{j_1}_{z_1} \cdots \nabla^{j_{2d}}_{z_{2d}} q_1(z) \right| \leq (2\pi)^{-2d} \int_{\R^{2d}} |w|^j \left| (1 - \Delta_w)^m \e^{-\phi(w)} \right| \, \dif w \leq C \kappa_1^{2m}.
		\]
		By interpolation, we obtain the desired estimate.
	\end{proof}
	
	Next, we focus on establishing the crucial two-sided sharp estimate for the large jump entering any cube that is far from the original point.
	For \( z = (x, v) \in \mathbb{R}^{2d} \) and \( r > 0 \), we define a cube in \( \mathbb{R}^{2d} \) and a ball in \( \mathbb{R}^d \) as follows:
	\begin{align} \label{Def-of-Q}
		Q_r(z) := \{(x', v') \in \R^{2d}: |x - x'| \leq r, \; |v - v'| \leq r\}, \quad B_r := \{x \in \mathbb{R}^d : |x| \leq r\}.
	\end{align}
	
	\begin{lemma} \label{lemma2}
		Under \eqref{Con3}, there exist two constants \( C_i = C_i(d, \alpha)>0, \, i = 0, 1 \) such that  
		\begin{align*}
			\P\left(Z_1^{(1)} \in Q_1(z)\right) \lesssim_{C_1} \frac{ \kappa_1^{3+2d+2\alpha}}{(|z| + 1)^{d + \alpha}} \int_0^1 \frac{\dif s}{(|\Gamma_s z| + 1)^{d + \alpha}}, \quad z \in \R^{2d},
		\end{align*}
		where \( \Gamma_s z := x - sv \) for \( z = (x, v) \), and for \( c_0 := \int_{|y| \geq 1} \frac{\dif y}{|y|^{d + \alpha}} \),
		\begin{align*}
			\P(Z_1^{(1)} \in Q_1(z)) \gtrsim_{C_0} \frac{\e^{-c_0 \kappa_1}}{(|z| + 1)^{d + \alpha}} \int_0^1 \left( |\Gamma_s z| + 1 \right)^{-d - \alpha} \, \dif s, \quad z \in \R^{2d}.
		\end{align*}
	\end{lemma}
	\begin{proof}
		In the following proof, all constants depend only on $d$ and $\alpha$. We fix a point \( z = (x,v) \in \mathbb{R}^{2d} \) and divide the proof into three steps.
		
		\textbf{(Step 1).} In this step we present a series representation for $\P(Z_1^{(1)}\in  Q_1(z))$.
		Let \( \{\tau_n, n \in \mathbb{N}\} \) and \( \{\xi_n, n \in \mathbb{N}\} \) be two independent families of i.i.d. random variables in \( [0, \infty) \) and \( \mathbb{R}^d \) with exponential distributions of parameter \( \lambda :=  \nu^{(1)}(\mathbb{R}^d) \) and \( \mu := \frac{ \nu^{(1)}}{\lambda} \), respectively. We set \( S_0^\tau = S_0^\xi = 0 \) and define 
		\[
		S_n^\tau := \tau_1 + \cdots + \tau_n, \quad S_n^\xi := \xi_1 + \cdots + \xi_n, \quad n \geq 1.
		\]
		Now, define 
		\[
		N_t := \max \{n: S_n^\tau \leq t\}, \quad H_t := S_{N_t}^\xi, \quad t > 0.
		\]
		Then, \( H_t \) is a compound Poisson process with L\'{e}vy measure \(  \nu^{(1)} \), which implies that 
		\[ 
		(H_t)_{t\geq 0} \stackrel{\mathrm{d}}{=} (L_t^{(1)})_{t\geq 0}.
		\] 
		Noticing that for any \( k \in \mathbb{N} \) and \( s > 0 \), 
		\[
		\{N_s = k\} = \{S_k^\tau \leq s < S_{k+1}^\tau\},
		\]
		we have
		\[
		\int_0^1 H_s \, \mathrm{d}s = \sum_{k=1}^\infty S_k^\xi \int_0^1 \1_{\{N_s = k\}} \, \mathrm{d}s = \sum_{k=1}^\infty S_k^\xi \left( \tau_{k+1} \1_{\{S_{k+1}^\tau \leq 1\}} + (1 - S_k^\tau) \1_{\{S_k^\tau \leq 1 < S_{k+1}^\tau\}} \right).
		\]
		Hence, for fixed $z\in\mR^{2d}$,
		recalling the definition of \( B_1 \) in \eqref{Def-of-Q}, we have the following series expression for the probability of \( Z^{(1)}_1 \) lying in \( Q_1(z) \):
		\begin{align*}
			&\P(Z_1^{(1)} \in Q_1(z)) = \P\left( \int_0^1 H_s \, \mathrm{d}s \in x + B_1, H_1 \in v + B_1 \right) \nonumber \\
			&= \sum_{n=1}^\infty \P\left( \sum_{k=1}^{n-1} S_k^\xi \tau_{k+1} + S_n^\xi(1 - S_n^\tau) \in x + B_1, S_n^\xi \in v + B_1, S_n^\tau \leq 1 < S_{n+1}^\tau \right)
			=:\sum_{n=1}^\infty I_n(z).
		\end{align*}
		
		Below, for the simplicity of notation, we write for  $n\in\mN$,
		$$
		\bs=(s_1,\cdots,s_n),\ \by=(y_1,\cdots,y_n),\ \Lambda_n:=\{\bs\in[0,1]^n: s_1+\cdots+s_n\leq 1\},
		$$
		and
		\begin{align}\label{eq108}
			\mathrm{d}\bs:=
			\mathrm{d}s_1\cdots \mathrm{d}s_{n},\ \   \bmu(\dif \by):=\mu(\dif y_1)\cdots\mu(\dif y_n)\in{\mathcal P}(\R^{nd}).
		\end{align}
		Since $\tau_k\sim \lambda\e^{-\lambda s}\1_{s\geq 0}$ and $\xi_k\sim\mu$ are independent,
		by elementary calculus, we have
		\begin{align}\label{eq18}
			I_n(z)&= \lambda^{n+1}\int_{[0,\infty)^{n+1} } \int_{\R^{nd}} \1_{\left\{\sum_{k=1}^{n-1}\sum_{j=1}^k y_j s_{k+1} + \sum_{j=1}^n y_j(1-s_1-\cdots -s_n)\in x+B_1\right\}}\1_{\left\{\sum_{j=1}^n y_j \in v+B_1\right\}} \nonumber\\
			&\qquad \times \1_{\left\{s_1+\cdots +s_n\leq 1<s_1+\cdots +s_{n+1} \right\}} \e^{-\lambda (s_1+\cdots +s_{n+1})} \mu(\mathrm{d}y_1)\cdots \mu(\mathrm{d}y_n)\dif s_1\cdots\dif s_n\dif s_{n+1}\nonumber\\
			& =  \lambda^{n}\e^{-\lambda}\int_{\Lambda_n} \int_{\R^{nd}} \1_{\left\{\sum_{k=1}^{n-1}\sum_{j=1}^k y_j s_{k+1} + \sum_{j=1}^n y_j(1-s_1-\cdots -s_n)\in x+B_1\right\}}\1_{\left\{\sum_{j=1}^n y_j \in v+B_1\right\}}  \bmu(\dif\by)\dif  \bs\nonumber\\
			&=  \lambda^{n}\e^{-\lambda}\int_{\Lambda_n} \int_{\R^{nd}} \1_{\left\{\sum_{k=1}^{n} y_k (1-\sum_{j=1}^k s_k) \in x+B_1\right\}}\1_{\left\{\sum_{j=1}^n y_j \in v+B_1\right\}} \bmu(\dif\by)\dif \bs.
		\end{align}
		Next, by the change of variable 
		$$
		(1-s_1,1-s_1-s_2,\cdots,1-s_1-\cdots-s_n)\to (s_1,s_2,\cdots,s_n),
		$$ 
		we further have
		$$
		I_n(z)= \lambda^{n}\e^{-\lambda}\int_{0\leq s_n\leq\cdots \leq s_1\leq 1} \int_{\R^{nd}} \1_{\left\{\sum_{j=1}^{n} y_j s_j \in x+B_1\right\}}\1_{\left\{\sum_{j=1}^n y_j \in v+B_1\right\}}  \bmu(\dif\by) \dif \bs.
		$$
		Since the multi-function
		$$
		(s_1,\cdots,s_n)\mapsto \int_{\R^{nd}} \1_{\left\{\sum_{j=1}^{n} y_j s_j \in x+B_1\right\}}\1_{\left\{\sum_{j=1}^n y_j \in v+B_1\right\}} \bmu(\dif\by)
		$$
		is symmetric under permutations, we get
		\begin{align}\label{eq118}
			I_n(z)= \frac{\lambda^n \e^{-\lambda}}{n!}\int_{[0,1]^n}\int_{\R^{nd}} \1_{\left\{\sum_{j=1}^{n} y_j s_j \in x+B_1\right\}}\1_{\left\{\sum_{j=1}^n y_j \in v+B_1\right\}}  \bmu(\dif\by)\mathrm{d}\bs.
		\end{align}
		
		\textbf{(Step 2).} In this step, we show the upper bound:
		\begin{align}\label{Goal-3}
			\P(Z_1^{(1)} \in Q_1(z)) = \sum_{n=1}^\infty I_n(z)
			\lesssim \frac{\kappa_1^{3+2d+2\alpha}}{(|z| + 1)^{d + \alpha}}  \int_0^1 \frac{\mathrm{d}s}{(|\Gamma_sz| + 1)^{d + \alpha}},
		\end{align}
		which is a direct consequence of the following estimate: for some \( C = C( d, \alpha)> 0 \) and any \( n \in \mathbb{N} \),
		\begin{align}\label{Goal-1}
			I_n(z) \lesssim_C \frac{\kappa_1^2 \lambda^{n-1} \e^{-\lambda}n^{2 + 2d + 2\alpha}}{n! (|z| + 1)^{d + \alpha}}  \int_0^1 \frac{\mathrm{d}s}{(|\Gamma_sz| + 1)^{d + \alpha}}.
		\end{align}
		Indeed, for $\beta\geq0$, define $f_\beta(s):= \sum_{n=1}^\infty \frac{n^\beta s^{n}}{n!},s>0$. It is easy to see that
		$f_0(s)= \e^{s}-1$ and for $k\in\mN$, $f_{k}(s)= s f_{k-1}'(s)$. By induction we have
		$$
		f_k(s)\leq C_ks^k \e^{s},\ \ s\geq1.
		$$
		Now for each $\beta>0$ with $\beta=k+r$, where $k\in\mathbb{N}_0$ and $r\in [0,1)$, we see that 
		$$
		f_\beta(s):=\sum_{n=1}^\infty \frac{n^\beta s^{n}}{n!} \leq s^{r} \sum_{n=1}^{[s]} \frac{n^k s^{n}}{n!} + s^{r-1} \sum_{n=[s]+1}^{\infty} \frac{n^{k+1} s^{n}}{n!}  \leq (C_k+C_{k+1})s^\beta \e^{s},\quad s\geq 1.
		$$
		Thus, if \eqref{Goal-1} is proven and let  $\beta:=2 + 2d + 2\alpha$,  then 
		\[
		\sum_{n=1}^\infty \frac{ \lambda^{n-1} \e^{-\lambda} n^{2 + 2d + 2\alpha}}{n! } =\frac{\e^{-\lambda}}\lambda\sum_{n=1}^\infty \frac{ \lambda^{n} n^\beta}{n! } =\frac{\e^{-\lambda} { f_\beta(\lambda)}}{\lambda}\lesssim {  \lambda^{\beta-1}}\lesssim \kappa_1^{1+2d+ 2\alpha}.
		\]
		Now, to achieve \eqref{Goal-1}, we first show the following \textit{claim}: there exists a constant \( C =   C( d, \alpha) > 0 \) such that for all \( n \geq 2 \),
		\begin{align}\label{Proof-of-Goal-1}
			I_n(z) \lesssim_C \frac{ \kappa_1 \lambda^{n-1} \e^{-\lambda} n^{1 + d + \alpha}}{n! (|z| + 1)^{d + \alpha}}  \int_{D_n} \1_{E_0}
			\boldsymbol{\mu}(\mathrm{d}\by) \mathrm{d} \bs,
		\end{align}
		where  $D_n:=[0,1]^n\times\R^{nd}$ and
		$$
		E_0:= \left\{(\bs,\by)\in D_n: \sum_{j=2}^{n} y_j (s_j-s_1) \in \Gamma_{s_1}z+B_2\right\}.
		$$
		{\it Proof of claim:} For the simplicity of notation, we write
		$$
		E_1:= \left\{(\bs,\by)\in D_n:\sum_{j=1}^n y_j \in v+B_1\right\},\ \ 
		E_2:= \left\{(\bs,\by)\in D_n: \sum_{j=1}^{n} y_j s_j \in x+B_1 \right\}.
		$$
		By the very definition, for  $(\bs,\by)\in E_1\cap E_2$, 
		\[
		\left|\sum_{j=2}^{n} y_j (s_j-s_1) -\Gamma_{s_1}z\right|\leq    \left|\sum_{j=1}^{n} y_j s_j - x\right| +s_1 \left| \sum_{j=1}^n y_j - v \right| \leq 2,
		\]
		which implies that 
		\begin{align}\label{EE0}
			E_1\cap E_2\subset E_0.
		\end{align}
		({\it Case $|z|<5$}): 
		By \eqref{EE0} and the fact that for $c_0:=\int_{|y|\geq 1}\dif y/|y|^{d+\alpha}$,
		\begin{align}\label{C0}
			c_0\leq \lambda= \nu^{(1)}(\mathbb{R}^d)\leq c_0 \kappa_1, 
		\end{align}
		by the definition \eqref{eq118}, we clearly have
		\begin{align} \label{eq: Proof-of-claim1}
			I_n(z) & \leq  \frac{\lambda^n \e^{-\lambda}}{n!}\int_{D_n} \1_{E_0}
			\boldsymbol{\mu}(\mathrm{d}\by) \mathrm{d} \bs\leq  \frac{ c_0\kappa_1 \lambda^{n-1} \e^{-\lambda}6^{d+\alpha}}{n!(|z|+1)^{d+\alpha}}\int_{D_n} \1_{E_0}
			\boldsymbol{\mu}(\mathrm{d}\by) \mathrm{d} \bs.
		\end{align}
		({\it Case $|z|\geq 5$}): Since $|\Gamma_1z|^2+|\Gamma_{-1}z|^2 = 2|z|^2$, we have 
		$$
		|\Gamma_{1}z|\vee|\Gamma_{-1}z|\geq |z|.
		$$ 
		Without loss of generality, we assume $|\Gamma_{-1}z|=|x+v|\geq|z|$.
		Since  on $E_1\cap E_2$, 
		$$
		|x+v|-\sum_{j=1}^n (1+s_j)|y_j|\leq\Bigg|\sum_{j=1}^n (1+s_j)y_j-x-v\Bigg|\leq 2,
		$$
		there exists  at least a $j\in \{1,\dots n\}$ such that 
		$$
		2|y_j|\geq (1+ s_j)|y_j|\geq \tfrac{1}{n}(|x+ v|-2)=\tfrac{1}{n}(|\Gamma_{1}z|-2)> \tfrac{1}{2n} (|z|+1),
		$$ 
		where the last inequality is due to $|\Gamma_{1}z|\geq|z|\geq5$.
		Therefore, by the definition \eqref{eq118}, 
		\begin{align} \label{eq: Proof-of-claim2}
			I_n(z) & \leq  \frac{\lambda^n\e^{-\lambda} }{n!} \sum_{j=1}^n \int_{D_n} \1_{E_1\cap E_2}  \1_{\{|y_j|> \frac{1}{4n}(|z|+1)\}}  \mu(\mathrm{d}y_1)\cdots \mu(\mathrm{d}y_n)\mathrm{d}\bs\nonumber\\
			& = \frac{\lambda^n n\e^{-\lambda}}{n!}  \int_{D_n} \1_{E_1\cap E_2} \1_{\{|y_1|> \frac{1}{4n}(|z|+1)\}} \1_{\{|y_1|\geq 1\}}\frac{\kappa(y_1)\mathrm{d} y_1}{\lambda|y_1|^{d+\alpha}}\cdots \mu(\mathrm{d}y_n)\mathrm{d}\bs\nonumber\\
			&\!\!\!\!\stackrel{\eqref{EE0}}{\leq } \frac{\lambda^n n\e^{-\lambda}}{n!}  \int_{D_n} 
			\1_{E_1\cap E_0}  \1_{\{|y_1|> \frac{1}{4n}(|z|+1)\}} \1_{\{|y_1|\geq 1\}}\frac{\kappa(y_1)\mathrm{d} y_1}{\lambda|y_1|^{d+\alpha}}\cdots \mu(\mathrm{d}y_n)\mathrm{d}{\bf s_n}\nonumber\\ 
			& \leq \frac{\kappa_1\lambda^{n-1} n (4n)^{d+\alpha}{\e^{-\lambda}}}{n! (|z|+1)^{d+\alpha}}  \int_{D_n} \1_{E_0}\1_{\left\{\sum_{j=1}^n y_j \in v+B_1\right\}}   \mathrm{d} y_1 \mu(\mathrm{d} y_2)\cdots \mu(\mathrm{d}y_n)\mathrm{d}\bs\nonumber\\
			& = \frac{\kappa_1\lambda^{n-1} n (4n)^{d+\alpha}{\e^{-\lambda}}|B_1|}{n! (|z|+1)^{d+\alpha}}  \int_{D_n} \1_{E_0}(\bs,\by)
			\bmu(\dif\by)\mathrm{d}\bs,
		\end{align}
		where the last step is due to Fubini's theorem and
		$$
		\int_{\mR^d}\1_{\left\{\sum_{j=1}^n y_j \in v+B_1\right\}}\dif y_1=|B_1|=|B_1|\int_{\mR^d}\mu(\dif y_1).
		$$
		Thus the assertion of the claim follows directly by \eqref{eq: Proof-of-claim1} and \eqref{eq: Proof-of-claim2}.
		
		Now, fix $s_1\in[0,1]$. If $|\Gamma_{s_1}z| >4$,	then on $E_0$,	there exists at least a 
		$j\in \{2,\dots, n\}$ such that 
		\[
		|y_j||s_j-s_1| >\tfrac{1}{n-1}(|\Gamma_{s_1}z| -2)> \tfrac{1}{4n}(|\Gamma_{s_1}z|+1),
		\]
		which implies that for any $s_2,\cdots s_{n}\in [0,1]$ and $\bs=(s_1,\cdots,s_n)\in[0,1]^n$,
		\begin{align}\label{Proof-of-Goal-2}
			&\int_{\R^{nd}} \1_{E_0}\bmu(\dif\by)\leq \sum_{j=2}^n  \int_{\R^{nd}} \1_{E_0\cap\left\{ |y_j||s_j-s_1| > \frac{1}{4n}(|\Gamma_{s_1}z|+1) \right\}} \mu(\mathrm{d}y_1)
			\mu(\mathrm{d}y_2)\cdots  \frac{\kappa(y_j)\mathrm{d} y_j}{\lambda |y_j|^{d+\alpha}}\cdots\mu(\mathrm{d}y_n) \nonumber\\
			&\quad\leq \sum_{j=2}^n \frac{\kappa_1|s_j-s_1|^{d+\alpha}(4n)^{d+\alpha}}{\lambda (|\Gamma_{s_1}z|+1)^{d+\alpha}} \int_{\R^{nd}} 
			\1_{\left\{\sum_{j=2}^n y_j (s_j-s_1) \in \Gamma_{s_1}z+B_2\right\} }\mu(\dif y_1) \mu(\mathrm{d}y_2)\cdots \mathrm{d} y_j\cdots \mu(\mathrm{d}y_n) \nonumber\\
			&\quad = \sum_{j=2}^n \frac{\kappa_1|s_j-s_1|^{d+\alpha}(4n)^{d+\alpha}}{\lambda (|\Gamma_{s_1}z|+1)^{d+\alpha}}\frac{|B_2|}{|s_j-s_1|^{d}} 
			\leq \frac{ \kappa_1 (4n)^{d+\alpha+1}}{\lambda (|\Gamma_{s_1}z|+1)^{d+\alpha}}\stackrel{\eqref{C0}}{\leq} \frac{ \kappa_1 n^{d+\alpha+1}}{c_0 (|\Gamma_{s_1}z|+1)^{d+\alpha}}.
		\end{align}
		If $|\Gamma_{s_1}z| \leq 4$, then we trivially have
		\begin{align}\label{Proof-of-Goal-3}
			& \int_{\R^{nd}} \1_{E_0}\bmu(\dif\by)\leq1
			\leq \frac{\kappa_1n^{d+\alpha+1} 5^{d+\alpha}}{(|\Gamma_{s_1}z|+1)^{d+\alpha}}.
		\end{align}
		Plugging \eqref{Proof-of-Goal-2} and \eqref{Proof-of-Goal-3} into \eqref{Proof-of-Goal-1}, we conclude that for any $n\geq 2$,
		\begin{align}\label{eq6}
			I_n(z) &\lesssim \frac{\kappa_1^2 \lambda^{n-1}\e^{-\lambda}n^{2(1+d+\alpha)}}{n! (|z|+1)^{d+\alpha}}  \int_{[0,1]^n}\frac{\mathrm{d}{\bs} }{(|\Gamma_{s_1}z|+1)^{d+\alpha}}
			\nonumber\\
			&= \frac{\kappa_1^2 \lambda^{n-1}\e^{-\lambda} n^{2+2d+2\alpha}}{n! (|z|+1)^{d+\alpha}}  \int_0^1 \frac{\dif s}{(|\Gamma_sz|+1)^{d+\alpha}}.
		\end{align}
		For $n=1$, it holds that 
		\begin{align*}
			I_1(z) &=\lambda \e^{-\lambda}\int_0^1 \int_{\R^{d}} \1_{\left\{ y s \in x+B_1\right\}}\1_{\left\{y \in v+B_1\right\}}  \mu(\mathrm{d}y)\mathrm{d}s\nonumber\\
			& \leq \lambda \e^{-\lambda}\int_0^1   \int_{\R^d}\1_{\{|\Gamma_sz|\leq 2\}}\1_{\{|y-v|\leq 1\}}\mu(\mathrm{d} y)\mathrm{d}s\nonumber\\
			& \leq \e^{-\lambda}\int_0^1 \1_{\{|\Gamma_sz|\leq 2\}}\mathrm{d}s  \int_{|y|\geq 1}\1_{\{|y-v|\leq 1\}}\frac{\kappa_1 \mathrm{d} y}{|y|^{d+\alpha}}.
		\end{align*}
		If $|x|>|v|+3$, then for any $s\in [0,1]$, $|\Gamma_sz|\geq |x|- s|v|>3$, this implies $I_1=0$.
		If $|x|\leq |v| +3$, then 	$|z| +1\leq 2|v|+4 \leq 4(|v|+1)$ and
		\begin{align}\label{eq7}
			I_1(z) \lesssim \frac{\kappa_1 \e^{-\lambda}\int_0^1 \1_{\{|\Gamma_sz|\leq 2\}}\mathrm{d}s  }{(|v|+1)^{d+\alpha}} \leq 
			\frac{\kappa_1 \e^{-\lambda}4^{d+\alpha}}{(|z|+1)^{d+\alpha}}
			\int_0^1 \frac{3^{d+\alpha}}{(|\Gamma_sz|+1)^{d+\alpha}} \mathrm{d}s.
		\end{align}
		Combining \eqref{eq6} and \eqref{eq7}, we get \eqref{Goal-1}.  
		
		\textbf{(Step 3).} In this step, we prove the lower bound: for some $C=C(d,\alpha)>0$,
		\begin{align}\label{Goal-2}
			\P(Z_1^{(1)}\in Q_1(z))\geq  \frac{C\e^{-\lambda}}{(|z|+1)^{d+\alpha}}  \int_0^1 \frac{\dif s}{(|\Gamma_sz|+1)^{d+\alpha}}.
		\end{align}
		Fix $r\in(0,1)$, whose value will be determined below. By the definition of $I_2(z)$, we have
		\begin{align}\label{eq9}
			&\P(Z_1^{(1)}\in Q_1(z))\geq  I_2(z)\geq \frac{\lambda^2 \e^{-\lambda}}2 \int_{D_2} \1_{\left\{y_1 s_1+y_2s_2 \in x+B_1\right\}}\1_{\left\{y_1+y_2 \in v+B_1\right\}}  \bmu(\dif\by)\dif\bs\nonumber\\
			&\quad\geq \frac{\lambda^2 \e^{-\lambda}}2 \int_{D_2} \1_{\left\{ y_2 (s_2-s_1) \in \Gamma_{s_1}z+B_{1-r}\right\}}\1_{\left\{y_1+y_2 \in v+B_r\right\}} \1_{\{\frac{1}{3}<|s_2-s_1| < \frac{2}{5}\}} \bmu(\dif\by)\dif\bs.
		\end{align}
		Notice that for any 	$|w|> 1$,
		\begin{align}\label{eq10}
			w(1+\tfrac{r}{2|w|})+ B_{r/4}\subset  w+B_r\cap \{|y|>1\}.
		\end{align}
		Indeed, if $y\in w(1+\frac{r}{2|w|})+ B_{r/4},$ then 
		\[
		|y-w|\leq \left|y-w\left(1+\tfrac{r}{2|w|}\right) \right|+ \left|w\left(1+\tfrac{r}{2|w|}\right)- w\right| \leq \tfrac{r}{4}+\tfrac{r}{2}<r
		\]
		and 
		\[
		|y| \geq \left|w\left(1+\tfrac{r}{2|w|}\right)\right| - \tfrac{r}{4}= |w|+\tfrac{r}{4}>1.
		\]
		Therefore, for	$|w|>1$, by \eqref{Con3} we have
		\begin{align*}
			\int_{\R^d} \1_{\{y \in w+ B_r \}} \mu(\mathrm{d} y)&=\int_{|y|>1} \1_{\{y \in w+ B_r \}}\frac{\kappa(y)\mathrm{d}y}{\lambda |y|^{d+\alpha }}\nonumber\\
			&\geq \int_{|y|>1} \frac{\1_{\{y \in w+ B_r \}}}{\lambda(|w|+r)^{d+\alpha}}\dif y\stackrel{\eqref{eq10}}{\geq} \frac{|B_{r/4}|}{\lambda(|w|+1)^{d+\alpha}}.
		\end{align*}
		Plugging this into \eqref{eq9} with $w=v-y_2$, we get
		\begin{align*}
			\P(Z_1^{(1)}\in Q_1(z))
			\geq \frac{\lambda \e^{-\lambda}|B_{r/4}|}2 \int_{[0,1]^2}\int_{\R^{d}} &
			\frac{\1_{\left\{|v-y_2|> 1 \right\}} }{(|v-y_2| +1)^{d+\alpha}}\1_{\left\{ y_2 (s_2-s_1) \in \Gamma_{s_1}z+B_{1-r}\right\}}\\
			&\times  \1_{\{\frac{1}{3}<|s_2-s_1| <\frac{2}{5}\}}  \mu(\mathrm{d}y_2)\dif s_1\dif s_2.
		\end{align*}
		Note that $|y_2 (s_2-s_1)-\Gamma_{s_1}z|\leq 1-r$ and $\frac{1}{3}<|s_2-s_1| <\frac{2}{5}$ 
		imply 
		$$
		\tfrac{|y_2|}{3} <|y_2| |s_2-s_1| \leq |\Gamma_{s_1}z| +1-r\leq 2|z|+1
		$$
		and
		$$
		|v-y_2|+1\leq|v|+|y_2|+1 \leq 7|z|+4.
		$$ 
		Therefore, 
		\begin{align}\label{eq16}
			\P(Z_1^{(1)}\in Q_1(z))
			\geq  \frac{\lambda \e^{-\lambda} |B_{r/4}|}{2(7|z| +4)^{d+\alpha}}\int_{[0,1]^2}  J(s_1, s_2)\1_{\{\frac{1}{3}<|s_2-s_1| <\frac{2}{5}\}}\dif s_1\dif s_2,
		\end{align}
		where for $s_1, s_2\in [0,1]$, 
		\[
		J(s_1, s_2):= \int_{\R^{d}} \1_{\left\{ y_2 (s_2-s_1) \in \Gamma_{s_1}z+B_{1-r}\right\}}  
		\1_{\left\{|v-y_2|> 1 \right\}} 
		\mu(\mathrm{d}y_2).
		\]
		By \eqref{Con3}, we have
		\begin{align*}
			J(s_1, s_2) &=  \int_{\R^d} \1_{\left\{ |y_2 (s_2-s_1)-\Gamma_{s_1}z|\leq 1-r\right\}} 
			\1_{\left\{|v-y_2|> 1\right\}}
			\1_{\{|y_2|\geq 1\}} \frac{\kappa(y_2)\mathrm{d}y_2}{\lambda|y_2|^{d+\alpha }}\nonumber\\
			&\geq \frac{|s_2-s_1|^{d+\alpha}}{\lambda \left( |\Gamma_{s_1}z| +1-r\right)^{d+\alpha}} \int_{\R^d} \1_{\left\{ |y_2 (s_2-s_1)-\Gamma_{s_1}z|\leq 1-r\right\}}
			\1_{\left\{|v-y_2|> 1\right\}}\1_{\{|y_2|\geq 1\}}\mathrm{d}y_2.
		\end{align*}
		Noting that $\1_{A_1\cap A_2\cap A_3}\geq \1_{A_1}-\1_{A_2^c}-\1_{A_3^c}$, we further have on  $\{\frac{1}{3}<|s_2-s_1| <\frac{2}{5}\}$,
		\begin{align*}
			J(s_1, s_2) 
			&\geq  \frac{|s_2-s_1|^{d+\alpha}}{\lambda \left( |\Gamma_{s_1}z| +1\right)^{d+\alpha}} \left(|B_{(1-r)/|s_2-s_1|}|-|B_1|-  |B_1|\right)\nonumber\\
			&= \frac{|s_2-s_1|^{d+\alpha}|B_1|}{\lambda\left( |\Gamma_{s_1}z| +1\right)^{d+\alpha}} \left(\left( \frac{1-r}{|s_2-s_1|}\right)^d-2\right)\\
			&\geq \frac{3^{-d-\alpha}|B_1|}{\lambda\left( |\Gamma_{s_1}z| +1\right)^{d+\alpha}} \left(\left( \frac{5(1-r)}{2}\right)^d-2\right).
		\end{align*}
		Plugging this into \eqref{eq16} and noticing that 
		$$
		\inf_{s_1\in[0,1]}\int_0^1 \1_{\{\frac{1}{3}< |s_2-s_1|< \frac{2}{5}\}}\mathrm {d}s_2\geq \frac{1}{15},
		$$ 
		we finally conclude
		$$
		\P(Z_1^{(1)}\in Q_1(z))\geq \frac{ \e^{-\lambda}|B_{r/4}| }{30(7|z| +4)^{d+\alpha}}\left( \int_{0}^1\frac{3^{-d-\alpha}|B_1|\dif s_1}{\left(|\Gamma_{s_1}z| +1\right)^{d+\alpha}}\right) 
		\left(\left( \frac{5(1-r)}{2}\right)^d-2\right).
		$$
		This completes the proof of \eqref{Goal-2} by choosing $r=1/10$. 
	\end{proof}
	
	We observe that the introduction of \( I_n(z) \) is pivotal to the proof of Lemma \ref{lemma2}. Next, we will employ a similar approach to demonstrate that the density function \( q_1(z) \) of \( Z_1^{(0)} \) is strictly positive and has a quantitative lower bound.
	First of all, we prepare the following simple lemma.
	\begin{lemma}\label{Le24}
		For any $n\geq 2$ and $u\in \R^d$ with $|u|\leq n$, there exists $u_1,\cdots u_{n}\in \R^d$ such that 
		\begin{align}\label{eq19}
			\frac{1}{3}\leq |u_i| \leq 1,\ 1\leq i\leq n \quad\mbox{and}\quad \sum_{i=1}^{n} u_i =u.
		\end{align}
	\end{lemma}
	\begin{proof}
		We use induction to prove \eqref{eq19}. 
		Recall that $\bar{u}=u/|u|.$
		For $n=2$, if $|u|\in (0,\tfrac{2}{3}]$, by taking
		$$
		u_1=-\bar u/3,\ \  u_2=u-u_1=\bar u(|u|+1/3),
		$$
		then it is easy to see that $|u_2|\in [\tfrac{1}{3},1]$.
		If $|u|\in [\tfrac{2}{3}, 2]$, it suffices to take $u_1=u_2=\tfrac{1}{2}u$.
		Now suppose that \eqref{eq19} is true for some $n\geq 2$. 
		For $|u|\leq n+1$, if we take $u_{n+1}=\bar u$, then
		$$
		|u-u_{n+1}|=||u|-1|\leq n.
		$$
		By induction, there are
		$u_1,\cdots u_n\in \R^d$ with $|u_i|\in[\frac13,1]$ such that $u_1+\cdots+u_n=u-u_{n+1}$. 
	\end{proof}
	We can now present the following sharp two-sided estimates for the density \( q_1(z) \) of the small jump part \( Z^{(0)}_1 \).
	\begin{lemma}
		If \( \kappa(x) \equiv 1 \), then there are constants \( \delta_0,\delta_1, C_0, C_1 > 0 \) such that for all \( z \in \mathbb{R}^{2d} \), 
		\begin{align}\label{Goal1}
			C_0\e^{-\delta_0 (|z| + 2) \log (|z| + 2)} \leq q_1(z) \leq C_1\e^{-\delta_1 (|z| + 2) \log (|z| + 2)}.
		\end{align}
	\end{lemma}
	\begin{proof} 
		{\bf (Lower bound).}
		By the L\'{e}vy-It\^o decomposition, assume that \( L_t^{(0)} = L_t^{(0,0)} + L_t^{(0,1)} \), where \( L_t^{(0,0)} \) and \( L_t^{(0,1)} \) are two independent L\'{e}vy processes with 
		characteristic exponents $\psi_{ \nu^{(0,0)}}$ and $\psi_{ \nu^{(0,1)}}$, respectively, where
		\[
		\nu^{(0,0)}(\mathrm{d}x) :=  \frac{\1_{\{|x| \leq { 1/3}\}}\mathrm{d}x}{|x|^{d+\alpha}}, \quad  \nu^{(0,1)}(\mathrm{d}x) :=  \frac{\1_{\{|x| \in (1/3, 1]\}}\mathrm{d}x}{|x|^{d+\alpha}}.
		\]
		Define
		\[
		Z_t^{(0,i)} := \left( \int_0^t L_s^{(0,i)} \, \mathrm{d}s, L_t^{(0,i)} \right),\ \ i=0,1.
		\]
		Let  $q_1^{(0)}$ be
		the density of \( Z_1^{(0,0)}\) (see Lemma \ref{lemma:small-jump}). 
		Since $Z^{(0)}_1=Z_1^{(0,0)}+Z_1^{(0,1)}$, by the independence,
		\begin{align}\label{GH8}
			q_1(z) = \mathbb{E} \left( q_1^{(0)}(z - Z_1^{(0,1)}) \right).
		\end{align}
		Note that by Lemma \ref{lemma1} and the Fourier inverse transform,
		\[
		q_1^{(0)}(0) = (2\pi)^{2d}\int_{\mathbb{R}^{2d}} \exp \left( -\int_0^1 \psi_{ \nu^{(0,0)}}(s\xi + \eta) \, \mathrm{d}s \right) \, \mathrm{d}\xi \, \mathrm{d}\eta>0. 
		\]
		Thus, by the continuity of $z\mapsto q_1^{(0)}(z)$ (see Lemma \ref{lemma:small-jump}), there exists a $\delta\in(0,\frac12)$ such that 
		\begin{align}\label{eq17}
			c_\delta:= \inf_{z\in Q_\delta(0)} q_1^{(0)}(z) >0.
		\end{align}
		
		Next we treat the large jump part $Z_1^{(0,1)}$ and prove \eqref{Goal1}.  
		Since \( L_t^{(0,1)} \) is a compound Poisson process with L\'{e}vy measure \(  \nu^{(0,1)} \), 
		by repeating the argument in \textbf{Step 1} of Lemma \ref{lemma2}, with \( \lambda \) 
		replaced by \( \lambda_1 :=  \nu^{(0,1)}(\mathbb{R}^d)\in(0,\infty) \) and \( \mu \) replaced 
		by \( \widetilde\mu := \nu^{(0,1)}  /\lambda_1\), we obtain that for any \( z = (x, v) \in \mathbb{R}^{2d} \) and any \( n \in \mathbb{N} \),
		\begin{align}\label{LB-q}
			q_1(z) \stackrel{\eqref{GH8}}{\geq} \inf_{z\in Q_\delta(0)} q_1^{(0)}(z)
			\mathbb{P}\left(Z_1^{(0,1)} \in Q_\delta(z)\right)
			= c_\delta \sum_{n=1}^\infty \frac{\lambda_1^{n} \e^{-\lambda_1}}{n!}I_n(z),
		\end{align}
		where,  with the same notations as in \eqref{eq108},
		\[
		I_n(z)= \int_{[0,1]^{n}}  \int_{\mathbb{R}^{nd}} \1_{\left\{ |\sum_{j=1}^{n} y_j s_j-x|\leq{\delta} \right\}} \1_{\left\{ |\sum_{j=1}^{n} y_j - v|\leq{\delta} \right\}} 
		\widetilde\bmu(\dif \by) \mathrm{d}\bs.
		\]
		Now we fix an integer $n$. Let \( \varepsilon_n := \frac{\delta}{4n} \). Since the support of \( \widetilde\mu \) is contained in \( \{ y : |y| \in (\frac13, 1] \} \),
		\begin{align}\label{eq22}
			I_{2n}(z)&\geq \varepsilon_n^n\int_{[0,1]^{n}}  \int_{\mathbb{R}^{2nd}} \1_{\left\{ |\sum_{j=1}^{n} y_j s_j-x|\leq{3\delta/4} \right\}} \1_{\left\{| \sum_{j=1}^{2n} y_j - v|\leq {\delta} \right\}} 
			\widetilde\bmu(\dif \by) \mathrm{d}\bs\nonumber \\
			&\geq \varepsilon_n^n\int_{[1-\varepsilon_n,1]^{n}}  \int_{\mathbb{R}^{2nd}} \1_{\left\{ |\sum_{j=1}^{n} y_j s_j - x|\leq{3\delta/4} \right\}} \1_{\left\{ |\sum_{j=1}^{2n} y_j - v|\leq{\delta} \right\}} 
			\widetilde\bmu(\dif \by) \mathrm{d}\bs\nonumber \\
			&\geq \varepsilon_n^{2n} \int_{\mathbb{R}^{2nd}} \1_{\left\{ |\sum_{j=1}^{n} y_j - x|\leq{\delta/2} \right\}} \1_{\left\{ |\sum_{j=1}^{2n} y_j - v| \leq{\delta} \right\}} 
			\widetilde\bmu(\dif \by)\nonumber \\
			&\geq \varepsilon_n^{2n} \int_{\mathbb{R}^{2nd}} \1_{\left\{ |\sum_{j=1}^{n} y_j - x|\leq{\delta/2} \right\}} \1_{\left\{ |\sum_{j=n+1}^{2n} y_j +x- v| \leq{\delta/2} \right\}} 
			\widetilde\bmu(\dif \by)\nonumber \\
			&=\varepsilon_n^{2n} J_n(x)J_n(v-x),
		\end{align}
		where the last step is due to Fubini's theorem, and
		$$
		J_n(x):=\int_{\mathbb{R}^{nd}} \1_{\left\{ |\sum_{j=1}^{n} y_j - x|\leq{\delta/2} \right\}} \widetilde\bmu(\dif \by).
		$$
		Now for $n\geq 2$ and $u\in\R^{d}$ with $|u|\leq n$,
		by Lemma \ref{Le24}, there are $n$-points $u_1,\cdots,u_n\in B_1\setminus B_{1/3}$
		such that 
		\begin{align*}
			J_n(u)&=\int_{\mathbb{R}^{nd}} \1_{\left\{ |\sum_{j=1}^{n} y_j - u_j|\leq{\delta/2} \right\}} 
			\widetilde\bmu(\dif \by)\geq \int_{\mathbb{R}^{nd}} \prod_{j=1}^n \1_{\left\{ | y_j - u_j|\leq{\delta/(2n)} \right\}} 
			\widetilde\bmu(\dif \by)\nonumber\\
			&=\prod_{j=1}^n \int_{B_1\setminus B_{1/3}} \1_{\left\{ | y_j - u_j|\leq{\delta/(2n)} \right\}} 
			\frac{\dif y_j}{\lambda_1|y_j|^{d+\alpha}}
			\geq \frac{1}{\lambda_1^n}\prod_{j=1}^n \int_{B_1\setminus B_{1/3}} \1_{\left\{ | y_j - u_j|\leq{\delta/(2n)} \right\}}\dif y_j.
		\end{align*}
		For each $j$, since $u_j\in B_1\setminus B_{1/3}$ and $\delta\in(0,\frac12)$, it is easy to see that there is a $\tilde u_j\in\R^d$ such that
		$$
		B_{\delta/(4n)}(\tilde u_j)\subset B_{\delta/(2n)}(u_j)\cap(B_1\setminus B_{1/3}).
		$$
		Thus, we further have
		\begin{align}\label{eq23}
			J_n(u)\geq \frac{1}{\lambda_1^n}\prod_{j=1}^n \int_{\mathbb{R}^{d}} \1_{B_{\delta/(4n)}(\tilde u_j)}\dif y_j
			= \left( \frac{\delta |B_1|}{4n\lambda_1}\right)^n.
		\end{align}
		Hence, for each $z\in\R^{2d}$, let us choose $n=2[|z|]+1$.
		By \eqref{LB-q}, \eqref{eq22} and \eqref{eq23}, we derive that
		\begin{align*}
			q_1 (z)& \stackrel{\eqref{LB-q}}{\geq} c_\delta \frac{\lambda_1^{2n} \e^{-\lambda_1}}{(2n)!}I_{2n}(z) \stackrel{\eqref{eq22}}{\geq} c_\delta \frac{\lambda_1^{2n} \e^{-\lambda_1}}{(2n)!} \left( \frac{\delta}{4n}\right)^{2n} J_n(x)J_n(v-x) \nonumber\\
			& \stackrel{\eqref{eq23}}{\geq}c_\delta  \frac{\lambda_1^{2n} \e^{-\lambda_1}}{(2n)!} \left( \frac{\delta}{4n}\right)^{2n} { \left( \frac{\delta |B_1|}{4n\lambda_1}\right)^{2n}} 
			=c_\delta  \frac{\e^{-\lambda_1}}{(2n)!} {  \left(\frac{\delta^2|B_1|}{16 n^2}\right)^{2n}}.
		\end{align*}
		From this, we derive the lower bound by the fact that $(2n)!\leq (2n)^{2n}$.
		
		{\bf (Upper bound).}	
		Since the support of $ \nu^{(0)}(\dif x)= \1_{\{|x| \leq 1\}}\mathrm{d}x/|x|^{d+\alpha}$ is contained in the unit ball,
		by \cite[Theorem 26.1]{Sato}, we have for any $a\in(0,1)$,
		$$
		\sup_{s\in[0,1]}\mE \e^{a |L_s^{(0)}|\log|L_s^{(0)}|}<\infty.
		$$
		In particular, if we define $\varphi_a(r):= \e^{a (r+2)\log (r+2)}$ for $a, r\geq 0$, then for any $a\in (0,1)$, 
		\begin{align}\label{Claim}
			\sup_{s\in [0,1]} \E\varphi_a (|L_s^{(0)}|)<\infty.
		\end{align}	
		Since $\varphi_a$ is  increasing and convex, 
		by Jensen's inequality, we see that for any $a\in (0,1)$, 
		\begin{align}
			\sup_{t\in [0,1]}	\E\varphi_a\left(\int_0^t |L_s^{(0)}|\mathrm{d}s\right) \leq	\E\varphi_a\left( \int_0^1 |L_s^{(0)}|\mathrm{d}s\right)\leq \int_0^1 \E \varphi_a(|L_s^{(0)}|)\mathrm{d}s<\infty. 
		\end{align}
		Therefore, by $\varphi_a(|z|)\leq(\varphi_a(2|x|)+\varphi_a(2|v|))/2$, we have for any  $a\in (0,1/2)$,
		\begin{align}\label{eq4}
			\sup_{t\in [0,1]}\E \varphi_a\left(|Z_t^{(0)}|\right) & \leq  \sup_{t\in [0,1]}\frac12	\E\varphi_a\left(2\int_0^t |L_s^{(0)}|\mathrm{d}s\right)  + 	\sup_{t\in [0,1]}
			\frac12 \E\varphi_a (2|L_t^{(0)}|) \nonumber\\
			&\leq C_{a} \sup_{t\in [0,1]}	\E\varphi_{2a}\left(\int_0^t |L_s^{(0)}|\mathrm{d}s\right)  +  C_{a}\sup_{t\in [0,1]} \E\varphi_{2a} (|L_t^{(0)}|) <\infty,
		\end{align}
		where in the last inequality we used inequality $\varphi_a(2r)\leq C_a\varphi_{2a}(r)$.
		Noticing that 
		\begin{align*}
			Z_1^{(0)}=\left(\int_0^{1/2}(L_{s+1/2}^{(0)}-L_{1/2}^{(0)})\mathrm{d}s +\int_0^{1/2}L_s^{(0)}\mathrm{d}s +\frac{1}{2}L_{1/2}^{(0)}, L_1^{(0)}-L_{1/2}^{(0)}+L_{1/2}^{(0)} \right),
		\end{align*}
		by the independence of the increments, we have 
		\begin{align*}
			q_1(z) = \int_{\R^{2d}} q_{1/2} (z') q_{1/2}(z-\tilde{z}')\mathrm{d}z'=: I_1+I_2,
		\end{align*}
		where $\tilde{z}':=(x'+\frac{1}{2}v', v')$ and 
		\begin{align*}
			I_1 &:=  \int_{|z-\tilde{z}'| \leq |z|/2} q_{1/2} (z') q_{1/2}(z-\tilde{z}')\mathrm{d}z',\nonumber\\
			I_2 &:=  \int_{|z-\tilde{z}'| > |z|/2} q_{1/2} (z') q_{1/2}(z-\tilde{z}')\mathrm{d}z'.
		\end{align*}
		For $I_1$, since $|z-\tilde{z}'| \leq |z|/2$ implies that 
		\[
		\tfrac{1}{2}|z|\leq |\tilde{z}'| 
		=\sqrt{\left|x'+ \tfrac{1}{2}v'\right|^2 + |v'|^2 } \leq |z'|+ \tfrac{1}{2}|v'| \leq \tfrac{3}{2}|z'| \Rightarrow |z'|\geq \tfrac{1}{3}|z|.
		\]
		Therefore, it holds that 
		\begin{align}\label{eq2}
			I_1& \leq \Vert q_{1/2} \Vert_\infty 
			\int_{|z'|\geq \tfrac{1}{3}|z|}
			q_{1/2}(z')\mathrm{d}z'= \Vert q_{1/2} \Vert_\infty  \P\left(|Z_{1/2}^{(0)}|\geq \frac{1}{3}|z|\right).
		\end{align}
		For $I_2$, using the change of variable, we have
		\begin{align}\label{eq3}
			I_2 & = \int_{|z'| > |z|/2} q_{1/2}\left(x-x'-\frac{v-v'}{2}, v-v' \right)q_{1/2}(z')\mathrm{d}z'\leq \Vert q_{1/2} \Vert_\infty  \P\left(|Z_{1/2}^{(0)}|\geq \frac{1}{2}|z|\right).
		\end{align}
		Combining \eqref{eq2}, \eqref{eq3} and the Markov inequality, fixing any $a\in (0,1/2)$, we conclude that 
		\begin{align}\label{eq5}
			q_1(z)\leq 2 \Vert q_{1/2} \Vert_\infty  \P\left(|Z_{1/2}^{(0)}|\geq \frac{1}{3}|z|\right)\leq \frac{2 \Vert q_{1/2} \Vert_\infty }{ \varphi_a(|z|/3)}\E   \varphi_a \left(| Z_{1/2}^{(0)}|\right).
		\end{align}
		Combining \eqref{eq4} and \eqref{eq5}, we get the upper bound. 
	\end{proof} 
	
	Now we show the following quantitative lower bound estimate for the general small jump part.
	
	\begin{lemma}\label{lemma:small-jump-lowerbound}
		Under \eqref{Con3}, there is a constant \( C_1=C_1(d,\alpha) > 0 \) such that for \( \hat{c} := \int_{|y| \in (1/3, 1]} \frac{\dif y}{|y|^{d+\alpha}} \),
		\begin{align}\label{KD1}
			\inf_{z \in Q_1(0)} q_1(z)\geq C_1 \e^{-\hat{c} \kappa_1}.
		\end{align}
	\end{lemma}
	\begin{proof}
		By the L\'evy-It\^o decomposition, assume that \( L_t^{(0)} = L_t^{(0,0)} + L_t^{(0,1)} \), where \( L_t^{(0,0)} \) and \( L_t^{(0,1)} \) are two independent L\'evy processes with characteristic exponents \(\psi_{\nu_{0,0}}\) and \(\psi_{\nu_{0,1}}\), respectively, where the measures \(  \nu_{0,0}\) and \(\nu_{0,1}\) are defined as:
		\[
		\nu_{0,0}(\mathrm{d}x) := \1_{\{|x| \leq 1\}} \frac{\mathrm{d}x}{|x|^{d+\alpha}} , \quad \nu_{0,1}(\mathrm{d}x) := \1_{\{|x| \leq 1\}}  \frac{\kappa(x) - 1}{|x|^{d+\alpha}} \, \mathrm{d}x. 
		\]
		Define
		\[
		Z_t^{(0,i)} := \left( \int_0^t L_s^{(0,i)} \, \mathrm{d}s, L_t^{(0,i)} \right),\ \ i=0,1.
		\]
		Let \(q_1^{(0)}\) be the density of \( Z_1^{(0,0)} \). Since \( Z^{(0)}_1 = Z_1^{(0,0)} + Z_1^{(0,1)} \), by independence, we have
		\begin{align}\label{GH9}
			q_1(z) = \mathbb{E} \left( q_1^{(0)}(z - Z_1^{(0,1)}) \right).
		\end{align}
		Let \( N(E,t) := \sum_{s\leq t} \1_E(L_s^{(0,1)} - L_{s-}^{(0,1)}) \) be the Poisson random measure associated with \( L_t^{(0,1)} \). By the L\'evy-It\^o decomposition and \eqref{DQ2}, we can write
		\[
		L_t^{(0,1)} = \int_0^t \int_{\R^d \setminus \{0\}} x \widetilde{N}(\mathrm{d}x, \mathrm{d}s),
		\]
		where \( \widetilde{N}(\mathrm{d}x, \mathrm{d}s) = N(\mathrm{d}x, \mathrm{d}s) -  \nu_{0,1}(\mathrm{d}x) \mathrm{d}s \) is the compensated Poisson random measure. By It\^o's isometry, we get
		\begin{align*}
			\mathbb{E}\left(\left|L_t^{(0,1)}\right|^2\right) 
			= t \int_{\R^d \setminus \{0\}} |x|^2  \nu_{0,1}(\mathrm{d}x
			\leq t \kappa_1 \int_{\{0 < |x| \leq 1\}} \frac{|x|^2}{|x|^{d+\alpha}} \, \mathrm{d}x =: C_2 t \kappa_1.
		\end{align*}
		Thus, for any \( R > 0 \), by Markov's inequality and \( |y - x|^2 \leq 2|y|^2 + 2|x|^2 \), for any \( x, v \in \R^d \), we have
		\begin{align}\label{eq24}
			\mathbb{P}\left(\left| \int_0^1 L_s^{(0,1)} \, \mathrm{d}s - x \right| > R\right) 
			& \leq \frac{2}{R^2} \left( |x|^2 + \mathbb{E}\left(\left| \int_0^1 L_s^{(0,1)} \, \mathrm{d}s \right|^2 \right) \right) \leq \frac{2}{R^2} \left( |x|^2 + C_2 \kappa_1 \right),
		\end{align}
		and
		\begin{align}\label{eq25}
			\mathbb{P}\left( \left| L_1^{(0,1)} - v \right| > R \right) 
			\leq \frac{2}{R^2} \left( |v|^2 + \mathbb{E}\left(\left| L_1^{(0,1)} \right|^2 \right) \right)
			\leq \frac{2}{R^2} \left( |v|^2 + C_2 \kappa_1 \right).
		\end{align}
		Now, for given \( z = (x,v) \in \R^{2d} \) and \( R > 0 \), combining \eqref{Goal1}, \eqref{GH9}, \eqref{eq24} and \eqref{eq25}, we get
		\begin{align*}
			q_1(z) 
			& \geq \inf_{w \in Q_R(0)} q_1^{(0)}(w) \, \mathbb{P}(Z_1^{(0,1)} \in Q_R(z))\geq C_0 \e^{-\delta_0(R+2)\log (R+2)} \mathbb{P}(Z_1^{(0,1)} \in Q_R(z)) \\
			& \geq C_0 \e^{-\delta_0 (R+2)\log (R+2)} \left( 1 - \mathbb{P}\left(\left| \int_0^1 L_s^{(0,1)} \, \mathrm{d}s - x \right| > R \right) - \mathbb{P}\left( \left| L_1^{(0,1)} - v \right| > R \right) \right) \\
			& \geq C_0 \e^{-\delta_0 (R+2)\log (R+2)} \left( 1 - \frac{2}{R^2}(|z|^2 + 2C_2 \kappa_1) \right).
		\end{align*}
		In particular, for \( z \in Q_1(0) \), by taking \( R = \sqrt{8(1 + C_2 \kappa_1)} \), we obtain
		\begin{align*}
			\inf_{z \in Q_1(0)} q_1(z) 
			\geq C_0 \e^{-\delta_0 (R+2)\log (R+2)} \inf_{z \in Q_1(0)} \left( 1 - \frac{2}{R^2}(|z|^2 + 2C_2 \kappa_1) \right) 
			\geq \frac{C_0}{2} \e^{-\delta_0 (R+2)\log (R+2)},
		\end{align*}
		which in turn implies \eqref{KD1}  since $(R+2)\log (R+2)= o(\kappa_1)$ as $\kappa_1\to\infty$.
	\end{proof}
	The right-hand side of the inequality in Lemma \ref{lemma2} involves an integral, which can be challenging to work with directly. To address this, we introduce a comparable quantity that is crucial for establishing the upper bound estimate for the heat kernel. This new quantity is more explicit and easier to handle, facilitating our analysis.
	\begin{lemma}\label{lemma3}
		For  $\beta>1$, there is a constant $C=C(d,\beta)>1$ such that
		for all $z=(x,v)\in \R^{2d}$, 
		\begin{align}\label{Goal-4}
			\int_0^1 \frac{\dif s}{(|\Gamma_sz|+1)^{\beta}}  \asymp_C \frac{\left(\inf_{s\in [0,1]}|\Gamma_sz| +1 \right)^{1-\beta}}{|z|+1}=:\frac{\cM_\beta(z)}{|z|+1}.
		\end{align}
	\end{lemma}
	\begin{proof} 
		In the following proof, all the constants only depend on $d,\beta$. We fix $z=(x,v)\in\R^{2d}$.
		When $|z|<3$, since $|\Gamma_sz|\leq 2|z|$ for $s\in[0,1]$, it holds that 
		\[
		\int_0^1 \frac{\dif s}{(|\Gamma_sz|+1)^{\beta}} \asymp 1\asymp \frac{\cM_\beta(z)}{|z|+1}.
		\]
		Now in	the following proof, we assume that $|z|\geq 3$. 
		If $|v|\leq \frac{|z|}3$, then $$|x| =\sqrt{|z|^2 -|v|^2}\geq \tfrac{2\sqrt{2}}{3}|z|\geq \tfrac{2}{3}|z|,$$ 
		which implies that for any $s\in [0,1]$,
		$$
		\tfrac{1}{3}|z| \leq |\Gamma_sz| \leq 2|z|.
		$$
		In this case, we see that when $|v|\leq |z|/3$,
		$$			\int_0^1 \frac{\dif s}{(|\Gamma_sz|+1)^{\beta}} \asymp \frac{1}{(|z|+1)^{\beta}} \asymp \frac{1}{|z|+1} \frac{1}{\left(\inf_{s\in [0,1]}|\Gamma_sz| +1 \right)^{\beta-1}}
		=\frac{\cM_\beta(z)}{|z|+1}.
		$$
		Below we assume that 
		\begin{align}\label{AA3}
			|z|\geq |v|\geq 	\tfrac{1}{3}|z|	\geq 1. 
		\end{align}
		Let 
		$$
		x= y + s_0 v,\ \ s_0=\<x,v\>/|v|^2,\ \ \<y,v\>=0.
		$$ 
		Then we have the following estimate:
		\begin{align}\label{eq11}
			|\Gamma_sz|+1 \asymp \sqrt{|\Gamma_sz|^2+1} = \sqrt{|y|^2+ |s-s_0|^2 |v|^2 +1}\asymp |y|+ |s-s_0||v| +1.
		\end{align}
		We consider three cases:
		
		({\it Case 1: $s_0<0$.}) In this case, by \eqref{eq11} we have
		\begin{align}\label{Case1-3}
			\inf_{s\in [0,1]}|\Gamma_sz| +1 \asymp |y|+ \inf_{s\in [0,1]} |s-s_0||v| +1 = |y|+  |s_0||v| +1.
		\end{align}
		Combining \eqref{eq11} and the change of variable, it holds that 
		\begin{align}\label{Case1-1}
			& \int_0^1 \frac{\dif s}{(|\Gamma_sz|+1)^{\beta}} \asymp \int_0^1 \frac{\dif s}{(|y|+ (s-s_0)|v| +1)^{\beta}}
			= \frac{1}{|v|}\int_{|s_0||v| + |y|}^{(1+|s_0|)|v| + |y|} \frac{\mathrm{d} s}{(s+1)^{\beta}} \nonumber\\
			&\qquad= \frac{1}{(\beta-1)|v|} \left(\frac{1}{(|s_0||v| + |y|+1)^{\beta-1}} -\frac{1}{((1+ |s_0|)|v| + |y|+1)^{\beta-1}}\right).
		\end{align}
		Using \eqref{AA3}, we see that 
		\[
		|z| =\sqrt{ |v|^2+|x|^2} =\sqrt{ |y|^2 + |s_0|^2|v|^2 + |v|^2} \geq \tfrac{|s_0| |v| + |y|+1}{3},
		\]
		and
		\begin{align}\label{Case1-2}
			& (1+ |s_0|)|v| + |y|+1 \geq 	|s_0| |v| + |y|+1 +\tfrac{|z|}3\geq 	\tfrac{10}{9}
			(|s_0| |v| + |y|+1 ). 
		\end{align}
		Combining  \eqref{Case1-1} and \eqref{Case1-2}, we obtain that
		\begin{align*}
			\int_0^1 \frac{\dif s}{(|\Gamma_sz|+1)^{\beta}} 
			&\asymp \frac{1}{|z|+1} \left(\frac{1}{(|s_0||v| + |y|+1)^{\beta-1}} -\frac{1}{((1+ |s_0|)|v| + |y|+1)^{\beta-1}}\right) \nonumber\\
			&\asymp   \frac{1}{|z|+1} \frac{1}{(|s_0||v| + |y|+1)^{\beta-1}},
		\end{align*}
		which together with \eqref{Case1-3} yields \eqref{Goal-4}.
		
		({\it Case 2: $s_0>1$.}) In this case, by \eqref{eq11} we have
		\begin{align}\label{Case2-1}
			\inf_{s\in [0,1]}|\Gamma_sz| +1 \asymp |y|+ \inf_{s\in [0,1]} |s-s_0||v| +1 = |y|+  
			(s_0-1)|v|+1.
		\end{align}
		Similar to \eqref{Case1-1}, 	it follows from \eqref{eq11} that 
		\begin{align*}
			& \int_0^1 \frac{\dif s}{(|\Gamma_sz|+1)^{\beta}} \asymp \int_0^1 \frac{\dif s}{(|y|+ (s_0-s)|v| +1)^{\beta}}  = \frac{1}{|v|}\int_{(s_0-1)|v| + |y|}^{s_0|v| + |y|} \frac{\mathrm{d} s}{(s+1)^{\beta}} \nonumber\\
			&\quad=\frac{1}{(\beta-1)|v|} \left(\frac{1}{((s_0-1)|v| + |y|+1)^{\beta-1}} -\frac{1}{(s_0|v| + |y|+1)^{\beta-1}}\right).
		\end{align*}
		Using \eqref{AA3}, it holds that 
		\[
		|z| =\sqrt{ |v|^2+|x|^2} =\sqrt{ |y|^2 + |s_0|^2|v|^2 + |v|^2} \geq \tfrac{(s_0-1) |v| + |y|+1}2,
		\]
		and
		$$
		s_0|v| + |y|+1 \geq  (s_0-1)|v| + |y|+1 +\tfrac{|z|}{3} \geq \tfrac{7}{6}\left(  (s_0-1)|v| + |y|+1 \right).
		$$
		Using a similar argument as in case 1, we get \eqref{Goal-4}. 
		
		({\it Case 3: $s_0\in [0,1]$.}) In this case, it holds that 
		$$
		\inf_{s\in [0,1]}|\Gamma_sz| +1 = |x-s_0v| +1 = |y| +1	.
		$$
		According to \eqref{eq11}, 
		\begin{align}\label{Case3-1}
			& \int_0^1 \frac{\dif s}{(|\Gamma_sz|+1)^{\beta}} \asymp \int_0^1 \frac{\dif s}{(|y|+ |s-s_0||v| +1)^{\beta}}\nonumber\\
			&\quad = \frac{1}{|v|}\int_{|y|}^{s_0|v| + |y|} \frac{\mathrm{d} s}{(s+1)^{\beta}} + \frac{1}{|v|}\int_{ |y|}^{(1-s_0)|v| + |y|} \frac{\mathrm{d} s}{(s+1)^{\beta}} \\
			&\quad\leq 	\frac{2}{|v|}
			\int_{|y|}^{\infty} \frac{\mathrm{d} s}{(s+1)^{\beta}} 
			\stackrel{\eqref{AA3}}{\lesssim} \frac{1}{|z|+1} \frac{1}{(|y|+1)^{\beta-1}}=\frac{\cM_\beta(z)}{1+|z|}. \nonumber
		\end{align}
		For the lower bound, using the fact that $\max\{ s_0, 1-s_0\}\geq \frac{1}{2}$,
		by \eqref{Case3-1}, we also have 
		$$
		\int_0^1 \frac{\dif s}{(|\Gamma_sz|+1)^{\beta}} 
		\geq \frac{1}{|v|}\int_{|y|}^{\frac{1}{2}|v| + |y|} \frac{\mathrm{d} s}{(s+1)^{\beta}}=\frac{1}{(\beta-1)|v|}\left(\frac{1}{(|y| +1)^{\beta-1}}- \frac{1}{(\frac{1}{2}|v| + |y| +1)^{\beta-1}} \right).
		$$
		Recalling that $|v|\geq \frac{|z|}3 \geq \frac{1}{6}(|y|+1)$, by \eqref{AA3} we conclude
		\begin{align}\label{Case3-2}
			& \int_0^1 \frac{1}{(|\Gamma_sz|+1)^{\beta}} \gtrsim  \frac{1}{|z| +1}\frac{1}{(|y| +1)^{\beta-1}}\left(1- \left(\tfrac{6}{7}\right)^{\beta-1}\right).
		\end{align}
		Combining \eqref{Case3-1} and \eqref{Case3-2}, we get \eqref{Goal-4}. 
	\end{proof}
	
	Now we are in a position to give the proof of the main result.
	
	\begin{proof}[Proof of Theorem \ref{thm1}]
		By \eqref{Scaling}, it suffices to consider $p_1(z)=p^{\nu_t}_1(z)$ under the assumption \eqref{Con3}. 
		
		{\bf (Lower bound estimate).}
		Recall the definition of $Q_r(z)$ in \eqref{Def-of-Q} and \eqref{Nbeta}.
		By \eqref{Identity}, Lemmas  \ref{lemma2}, \ref{lemma:small-jump-lowerbound} and \ref{lemma3}, we have 
		\begin{align*}
			p_1(z) &=\E\left(q_1(z- Z_1^{(1)})\right) \geq \E\left(q_1(z- Z_1^{(1)})\1_{\left\{z-Z_1^{(1)}\in Q_1(0)\right\}}\right)\nonumber\\
			&\gtrsim \e^{-c_0\kappa_1}\inf_{w\in Q_1(0)} q_1(w)\frac{\cM_{d+\alpha}(z)}{(|z|+1)^{d+\alpha+1}}\gtrsim \e^{-c_0^* \kappa_1}  \cN_{d+\alpha}(z),
		\end{align*}
		where  $c_0^*= c_0+\hat{c}_0= \int_{|y|>1/3} \frac{\dif y}{|y|^{d+\alpha}}$. 
		
		{\bf (Upper bound and gradient estimates).} Fix $j \in \N_0$. By observing \eqref{Identity}, we have
		\begin{align}\label{upper-bound-1}
			\left|\nabla^{j}_{z} p_1(z)\right| \leq \E\left( \left|\nabla^{j}_{z} q_1(z - Z_1^{(1)})\right| \1_{\left\{z - Z_1^{(1)} \in Q_{|z|/2}(0)\right\}}\right) + 2 \sup_{|x'|\vee |v'| > |z|/2} \left|\nabla^{j}_{z'} q_1(z')\right|.
		\end{align}
		By Lemma \ref{lemma:small-jump}, we have
		$$
		2 \sup_{|x'|\vee |v'| > |z|/2} \left|\nabla^{j}_{z'} q_1(z')\right| \lesssim \frac{\kappa^{2d+2\alpha}_1}{(|z| + 1)^{2d + 2\alpha}}.
		$$
		Noting that $\inf_{s\in[0,1]} |\Gamma_s z| \leq 2|z|$, and using the definition \eqref{Goal-4} of $\cM_{d+\alpha}(z)$, we get
		$$
		2 \sup_{|x'|\vee |v'| > |z|/2} \left|\nabla^{j}_{z'} q_1(z')\right| \lesssim \frac{\kappa^{2d+2\alpha}_1 \cM_{d+\alpha}(z)}{(|z| + 1)^{d + \alpha + 1}} = {\kappa^{2d+2\alpha}_1}\cN_{d+\alpha}(z).
		$$
		Thus, to prove the upper bound and gradient estimates, it remains, by \eqref{upper-bound-1}, to show that
		\begin{align}\label{HA1}
			\sI := \E\left( \left|\nabla^{j}_{z} q_1(z - Z_1^{(1)})\right| \1_{\left\{z - Z_1^{(1)} \in Q_{|z|/2}(0)\right\}}\right) \lesssim \kappa_1^{4+4d+3\alpha}  \cN_{d+\alpha}(z).
		\end{align}
		To proceed, we choose a suitable $N$ and a set $\{z_i := (x_i, v_i) : 1 \leq i \leq N\} \subset Q_{|z|/2}(0)$ such that
		$$
		Q_{|z|/2}(0) \subset \bigcup_{i=1}^N Q_1(z_i),
		$$
		and for some $C_1,C_2$ only depending on $d$ and $\alpha$,
		\begin{align}\label{Selection}
			\sum_{i=1}^N \sup_{w \in Q_1(z_i)} \sqrt{\left|\nabla^{j}_{w} q_1(w)\right|}
			\leq \sum_{i=1}^N\frac{C_1\kappa_1^{d+1}}{(1+|z_i|)^{d+1}}\leq C_2\kappa_1^{d+1}.
		\end{align}
		Hence,
		\[
		\sI \leq \sum_{i=1}^N \E\left( \left|\nabla^{j}_{z} q_1(z - Z_1^{(1)})\right| \1_{\left\{z - Z_1^{(1)} \in Q_1(z_i)\right\}}\right)
		\leq \sum_{i=1}^N \sup_{w \in Q_1(z_i)} \left|\nabla^{j}_{w} q_1(w)\right| \P\left(z - Z_1^{(1)} \in Q_1(z_i)\right).
		\]
		Noting that by Lemma \ref{lemma:small-jump} again, 
		$$
		\sup_{w \in Q_1(z_i)}\left|\nabla^{j}_{w} q_1(w)\right| \lesssim \sup_{w \in Q_1(z_i)}
		\frac{\kappa^{2(d+\alpha-1)}_1}{(|w| + 1)^{2(d + \alpha - 1)}} \lesssim \frac{\kappa^{2(d+\alpha-1)}_1}{(|z_i| + 1)^{2(d + \alpha - 1)}},
		$$
		by Lemmas \ref{lemma2} and \ref{lemma3}, we further have
		\begin{align}\label{eq12}
			\sI &\lesssim \sum_{i=1}^N \sup_{w \in Q_1(z_i)} \sqrt{\left|\nabla^{j}_{w} q_1(w)\right|} \times \frac{\kappa^{d+\alpha-1}_1}{(|z_i| + 1)^{d + \alpha - 1}} \P\left(Z_1^{(1)} \in Q_1(z - z_i)\right)\nonumber\\
			&\lesssim \sum_{i=1}^N \sup_{w \in Q_1(z_i)} \sqrt{\left|\nabla^{j}_{w} q_1(w)\right|} \times \frac{\kappa^{3d+3\alpha+2}_1}{(|z_i| + 1)^{d + \alpha - 1}} \frac{
				\cM_{d+\alpha}(z - z_i)}{(|z - z_i| + 1)^{d + \alpha + 1}}.
		\end{align}
		Since $z_i \in Q_{|z|/2}(0)$, we have
		$$
		|z_i| = \sqrt{|x_i|^2 + |v_i|^2} \leq \tfrac{\sqrt{2}}{2} |z| \Rightarrow |z - z_i| + 1 \geq (1 - \tfrac{\sqrt{2}}{2}) |z| + 1,
		$$
		which, combined with \eqref{eq12}, gives
		\begin{align}\label{eq13}
			\sI \lesssim \frac{\kappa_1^{2+3d+3\alpha}}{(|z| + 1)^{d + \alpha + 1}} \sum_{i=1}^N \sup_{w \in Q_1(z_i)} \sqrt{\left|\nabla^{j}_{w} q_1(w)\right|} \times \frac{\cM_{d+\alpha}(z - z_i)}{(|z_i| + 1)^{d + \alpha - 1}}.
		\end{align}
		Define
		$$
		\inf_{s \in [0,1]} |\Gamma_s z| =: K_0.
		$$
		If $K_0 \leq 8$, then by \eqref{Selection} and \eqref{eq13},
		\begin{align*}
			\sI &\lesssim \frac{\kappa_1^{2+3d+3\alpha}}{(|z| + 1)^{d + \alpha + 1}} \sum_{i=1}^N \sup_{w \in Q_1(z_i)} \sqrt{\left|\nabla^{j}_{w} q_1(w)\right|} \leq \frac{C_2\kappa_1^{4+4d+3\alpha}}{(|z| + 1)^{d + \alpha + 1}} \nonumber\\
			&\leq \frac{C_2\kappa_1^{3+4d+3\alpha}}{(|z| + 1)^{d + \alpha + 1}} \frac{9^{d + \alpha}}{(K_0 + 1)^{d + \alpha + 1}} = C_2 {\kappa_1^{4+4d+3\alpha}} 9^{d + \alpha} \cN_{d+\alpha}(z).
		\end{align*}
		Now assume $K_0 > 8$. If $|z_i| > K_0 / 4$, then
		\begin{align}\label{eq14}
			\frac{\cM_{d+\alpha}(z - z_i)}{(|z_i| + 1)^{d + \alpha - 1}} \leq \frac{1}{(|z_i| + 1)^{ d + \alpha-1}} \lesssim \frac{1}{(K_0 + 1)^{ d + \alpha-1}} = \cM_{d+\alpha}(z).
		\end{align}
		If $|z_i| \leq K_0 / 4$, then
		$$
		\inf_{s \in [0,1]} |\Gamma_s(z - z_i)| \geq \inf_{s \in [0,1]} |\Gamma_s z| - \sup_{s \in [0,1]} |\Gamma_s z_i| \geq K_0 - 2|z_i| \geq \tfrac{K_0}{2}.
		$$
		In this case, by the definition of $K_0$ and $\cM_{d+\alpha}(z)=\left(\inf_{s\in [0,1]}|\Gamma_sz| +1 \right)^{1-d-\alpha}$, we have
		\begin{align}\label{eq15}
			\cM_{d+\alpha}(z - z_i) \lesssim \cM_{d+\alpha}(z).
		\end{align}
		Substituting \eqref{eq14} and \eqref{eq15} into \eqref{eq13}, we obtain
		\begin{align*}
			\sI & \lesssim \frac{\kappa_1^{2+3d+3\alpha} \cM_{d+\alpha}(z)}{(|z| + 1)^{d + \alpha + 1}} \sum_{i=1}^N \sup_{w \in Q_1(z_i)} \sqrt{\left|\nabla^{j}_{w} q_1(w)\right|} \nonumber\\
			&\leq \frac{C_2\kappa_1^{3+4d+3\alpha} \cM_{d+\alpha}(z)}{(|z| + 1)^{d + \alpha + 1}} =C_2\kappa_1^{4+4d+3\alpha}  \cN_{d+\alpha}(z).
		\end{align*}
		Thus, we have proven \eqref{HA1}, and hence
		$$
		\left|\nabla^{j}_{z} p_1(z)\right| \lesssim {\kappa_1^{3+4d+3\alpha}}  \cN_{d+\alpha}(z).
		$$
		For general $t > 0$ and $\kappa_0>0$, the result follows by \eqref{Scaling}. The proof is now complete by Lemma \ref{lemma3}.
	\end{proof}

	\section{Proof of Lemma \ref{Le11}}	
	\begin{proof}
		For $v\neq 0$, let $Q_v$ be the rotation matrix such that for $\bar v=v/|v|$, we have
		$$
		Q_v^T \bar v = (1,0,\cdots,0) =: \e_1.
		$$
		Since $Q_v$ is an orthogonal matrix, i.e., $Q_v Q_v^T = \mI$, for any $x = (x_1,\dots, x_d)^T \in \R^d$, we have
		\[
		\langle Q_v x, \bar v \rangle = \langle x, Q_v^T \bar v \rangle = x_1 \quad \text{and} \quad |Q_v x - v| = |x - Q_v^T v| = |x - |v|\e_1|.
		\]
		Thus, for each fixed $v \in \R^d \setminus \{0\}$ and $\beta > 0$, by definition we have
		\begin{align}
			\mathcal{N}_\beta(Q_v x, v)
			= \frac{1}{(1 + |z|)^{1 + \beta}} 
			\begin{cases}
				(1 + |x|)^{1 - \beta}, & x_1 < 0, \\
				(1 + \sqrt{|x|^2 - x_1^2})^{1 - \beta}, & 0 \leq x_1 \leq |v|, \\
				(1 + |x - |v|\e_1|)^{1 - \beta}, & x_1 > |v|.
			\end{cases}
		\end{align}
		Below, we frequently use the following simple fact: for $a, b \geq 0$,
		$$
		a \vee b \leq a + b \leq 2(a \vee b).
		$$
		Since $\det(Q_v) = 1$ and $|Q_v x| = |x|$, by a change of variables, we can write
		\begin{align}
			\int_{\R^d} |x|^q \mathcal{N}_\beta(x,v) \,\mathrm{d}x 
			&= \int_{\R^d} |Q_v x|^q \mathcal{N}_\beta(Q_v x, v) \,\mathrm{d}x \nonumber\\
			&= \int_{x_1 < 0} \frac{|x|^q}{(1 + |z|)^{1 + \beta}} (1 + |x|)^{1 - \beta} \,\mathrm{d}x \nonumber\\
			&\quad + \int_{x_1 \in [0, |v|]} \frac{|x|^q}{(1 + |z|)^{1 + \beta}} (1 + \sqrt{|x|^2 - x_1^2})^{1 - \beta} \,\mathrm{d}x \nonumber\\
			&\quad + \int_{x_1 > |v|} \frac{|x|^q}{(1 + |z|)^{1 + \beta}} (1 + |x - |v|\e_1|)^{1 - \beta} \,\mathrm{d}x \nonumber\\
			&=: H_\beta(v) + K_\beta(v) + L_\beta(v). \label{KM1}
		\end{align}
		For $H_\beta(v)$, using symmetry and polar coordinates, we have
		\begin{align*}
			H_\beta(v) &= \frac{1}{2} \int_{\R^d} \frac{|x|^q}{(1 + |z|)^{1 + \beta}} (1 + |x|)^{1 - \beta} \,\mathrm{d}x\asymp \int_0^\infty \frac{r^q}{(1 + |v| + r)^{1 + \beta}} (1 + r)^{1 - \beta} \,\mathrm{d}r \nonumber\\
			&\asymp \int_0^\infty \frac{r^q (1 \vee r)^{1 - \beta}}{((1 + |v|) \vee r)^{1 + \beta}} \,\mathrm{d}r\asymp \int_0^{1 + |v|} \frac{r^q (1 \vee r)^{1 - \beta}}{(1 + |v|)^{1 + \beta}} \,\mathrm{d}r + \int_{1 + |v|}^\infty r^{q - 2\beta} \,\mathrm{d}r.
		\end{align*}
		Since $q < 2\beta - 1$, we obtain
		\begin{align}
			H_\beta(v) \lesssim (1 + |v|)^{q - 2\beta + 1}. \label{KM2}
		\end{align}
		For $K_\beta(v)$, by a change of polar coordinates, we have
		\begin{align*}
			K_\beta(v)& \asymp \int_0^\infty \frac{r^{d-2}}{(1+r)^{\beta-1}}\dif r \int_{0}^{|v|} \frac{ (r+x_1)^{q}}{(1+|v|+ r+x_1)^{1+\beta}}  \dif x _1\nonumber\\
			&\asymp  \int_0^\infty \frac{r^{d-2}}{(1+r)^{\beta-1}(1+|v|+ r)^{1+\beta}}\dif r \int_{0}^{|v|} (r+x_1)^{q}\dif x _1\nonumber\\
			&=\int_0^\infty \frac{r^{d-2}|v|}{(1+r)^{\beta-1}(1+|v|+ r)^{1+\beta}} \int^1_0(r+s|v|)^q\dif s \dif r\\
			&\asymp \frac{|v|}{(1+|v|)^{1+\beta}}
			\int_0^{1+|v|} \frac{r^{d-2}}{(1+r)^{\beta-1}} \int^1_0(r+s|v|)^q\dif s \dif r\\
			&\quad+|v|\int_{1+|v|}^\infty r^{-2-2\beta+d} \int^1_0(r+s|v|)^q\dif s \dif r
			=:I_1+I_2.
		\end{align*}
		For $I_1$,  for 
		$r<1+|v|$ and $|v|\geq 1$, 
		since $\int^1_0(r+s|v|)^{q}\dif s\asymp (1+|v|)^{q},$
		we have
		\begin{align*}
			I_1&\asymp \frac{|v|(1+|v|)^{q}}{(1+|v|)^{1+\beta}}
			\int_0^{1+|v|} \frac{r^{d-2}}{(1+r)^{\beta-1}}\dif r\\
			&\asymp(1+|v|)^{q-\beta}
			\left[\int_0^1 r^{d-2}\dif r+ \int^{|v|+1}_1r^{d-1-\beta}\dif r\right]
			\asymp (1+|v|)^{q-\beta}.
		\end{align*}
		For $I_2$,  by $q<2\beta-d+1$, we have for $|v|\geq 1$,
		$$
		I_2 \asymp |v| \int_{1+|v|}^\infty r^{-2-2\beta+d+q} \dif r \asymp (1+|v|)^{q+d-2\beta}.
		$$
		Hence,
		\begin{align}\label{KM3}
			K_\beta(v) \asymp(1+|v|)^{q-\beta}.
		\end{align}
		Finally, we consider the term \( L_\beta(v) \). By changing to polar coordinates, we have
		\begin{align}\label{Asymp6}
			L_{\beta}(v) & \asymp \int_0^\infty \int_{|v|}^\infty \frac{(r+x_1 )^{q} r^{d-2}}{(1+|v|+r+x_1)^{1+\beta} (1+r+x_1-|v|)^{\beta-1}} \, \mathrm{d}x_1 \, \mathrm{d}r \nonumber \\
			& = \int_0^\infty r^{d-2} \, \mathrm{d}r \int_0^\infty \frac{(r+|v|+s)^{q}}{(1+2|v|+r+s)^{1+\beta} (1+r+s)^{\beta-1}} \, \mathrm{d}s \nonumber \\
			& \asymp \int_0^\infty r^{d-2} \, \mathrm{d}r \int_0^\infty \frac{((r+|v|)\vee s)^{q}}{((1+2|v|+r)\vee s)^{1+\beta} ((1+r)\vee s)^{\beta-1}} \, \mathrm{d}s.
		\end{align}
		If \( q = 2\beta - d \), then
		\begin{align*}
			L_{\beta}(v) & \gtrsim \int_{1+|v|}^\infty r^{d-2} \, \mathrm{d}r \int_r^\infty \frac{((r+|v|)\vee s)^{q}}{((1+2|v|+r)\vee s)^{1+\beta} ((1+r)\vee s)^{\beta-1}} \, \mathrm{d}s \nonumber \\
			& \asymp \int_{1+|v|}^\infty r^{d-2} \, \mathrm{d}r \int_r^\infty s^{-d} \, \mathrm{d}s \stackrel{s=ru}{=} \int_{1+|v|}^\infty r^{-1} \, \mathrm{d}r \int_1^\infty u^{-d} \, \mathrm{d}u = \infty.
		\end{align*}
		This implies \eqref{AG6}.
		
		Next, we assume \( q < 2\beta - d \). For simplicity, we define:
		\[
		a := 1 + r, \quad b := |v| + r, \quad c := 1 + 2|v| + r,
		\]
		and
		\[
		\mathcal{R}(r) := \int_0^\infty \frac{(b \vee s)^q}{(c \vee s)^{1+\beta} (a \vee s)^{\beta-1}} \, \mathrm{d}s.
		\]
		We can then split \( \mathcal{R}(r) \) into parts:
		\begin{align*}
			\mathcal{R}(r) &= \left( \int_0^a + \int_a^b + \int_b^c + \int_c^\infty \right) \frac{(b \vee s)^q}{(c \vee s)^{1+\beta} (a \vee s)^{\beta-1}} \, \mathrm{d}s \\
			&= b^q c^{-1-\beta} a^{2-\beta} + b^q c^{-1-\beta} \int_a^b s^{1-\beta} \, \mathrm{d}s \\
			&\quad + c^{-1-\beta} \int_b^c s^{q+1-\beta} \, \mathrm{d}s + \int_c^\infty s^{q-2\beta} \, \mathrm{d}s \\
			&\lesssim b^q c^{-1-\beta} a^{2-\beta} + b^{q+2-\beta} c^{-1-\beta} + c^{q+1-2\beta}.
		\end{align*}
		Since \( q < 2\beta - d \Rightarrow q + 1 - 2\beta < 0 \), for \( |v| \geq 1 \), we have:
		\begin{align*}
			\mathcal{R}(r) &\lesssim (|v| + r)^q (1 + 2|v| + r)^{-1-\beta} \left[ (1 + r)^{2-\beta} + (|v| + r)^{2-\beta} \right] + (1 + 2|v| + r)^{q+1-2\beta} \\
			&\lesssim (|v| + r)^{q-1-\beta} (1 + r)^{2-\beta} + (|v| + r)^{q+1-2\beta}.
		\end{align*}
		Thus, for \( |v| \geq 1 \),
		\begin{align}
			L_\beta(v) &\lesssim \int_0^\infty r^{d-2} \left[ (|v| \vee r)^{q-1-\beta} (1 \vee r)^{2-\beta} + (|v| \vee r)^{q+1-2\beta} \right] \, \mathrm{d}r \nonumber \\
			&\lesssim |v|^{q-1-\beta} \left[ 1 + \int_1^{|v|} r^{d-2} \left( r^{2-\beta} + |v|^{2-\beta} \right) \, \mathrm{d}r \right] + \int_{|v|}^\infty r^{d+q-1-2\beta} \, \mathrm{d}r \lesssim |v|^{q-1-\beta}. \label{KM4}
		\end{align}
		Combining \eqref{KM1}, \eqref{KM2}, \eqref{KM3}, and \eqref{KM4}, we obtain \eqref{Ineq-in-remark2}.
	\end{proof}

\end{document}